\documentclass[11pt,a4paper]{article}

\usepackage{ifthen}
\usepackage[utf8]{inputenc}
\usepackage[english]{babel}
\usepackage{amsmath,amsfonts,amssymb,amsthm}
\usepackage{graphicx}
\usepackage{bbding}
\usepackage{xcolor}
\usepackage{mathtools}
\usepackage{enumerate}
\usepackage[norelsize]{algorithm2e}
\usepackage{tikz}
\usepackage{wasysym}

\newtheorem{theorem}{Theorem}
\newtheorem*{theorem*}{Theorem}
\newtheorem{lemma}[theorem]{Lemma}
\newtheorem{claim}[theorem]{Claim}

\theoremstyle{definition}
\newtheorem{definition}[theorem]{Definition}

\theoremstyle{remark}

\newcommand{\ignore}[1]{} 

\newcommand{\cC}{\ensuremath{\mathcal C}}

\newcommand{\cF}{\ensuremath{\mathcal F}}

\newcommand{\cI}{\ensuremath{\mathcal I}}

\newcommand{\cP}{\ensuremath{\mathcal P}}

\newcommand{\cT}{\ensuremath{\mathcal T}}

\newcommand{\NN}{\ensuremath{\mathbb N}}

\newcommand{\PP}{\ensuremath{\mathbb P}}

\newcommand{\eps}{\varepsilon}
\renewcommand{\phi}{\varphi}

\newcommand{\abs}[1]{\lvert#1\rvert}

\renewcommand{\P}{\mathrm{P}}
\newcommand{\polylog}{\mathrm{polylog}}

\newcommand{\Hknp}{H^k(n, p)}

\newcommand{\lmax}{t_{\text{max}}}
\newcommand{\dmax}{d_{\text{max}}}
\newcommand{\fo}{\textup{fo}}
\newcommand{\reg}{\textup{reg}}
\newcommand{\flower}{\text{\tiny \SixFlowerPetalDotted}}

\newcommand{\HFi}{K_3^{+k}}
\newcommand{\HFii}{C_{t_k}}
\newcommand{\Tlk}{T^k_{\ell}}

\newcommand{\binoms}[2]{{\textstyle\binom{#1}{#2}}} 
\renewcommand{\phi}{\varphi}

\renewcommand{\P}{\PP}

\usepackage{appendix}

\date{}
\title{\vspace{-0.7cm}Symmetric and asymmetric Ramsey properties in random hypergraphs}
\author{
	Luca Gugelmann \thanks{
		Institute of Theoretical Computer Science, ETH Zurich, 8092 Zurich, Switzerland. \newline
		Email: {\tt lgugelmann|nskoric|steger|hthomas@inf.ethz.ch}.
	}
	\and
	Rajko Nenadov \thanks{
		School of Mathematical Sciences, Monash University, Melbourne, Australia. \newline 
		Email: {\tt rajko.nenadov@monash.edu}.
	} 
	\and
	Yury Person \thanks{
		Institute of Mathematics, Goethe-Universit\"at, 60325 Frankfurt am Main, Germany. \newline
		Email: {\tt person@math.uni-frankfurt.de}. The author was supported by DFG grant PE 2299/1-1.
	} 
	\and
	Nemanja \v Skori\'c \footnotemark[1]
	\and
	Angelika Steger \footnotemark[1]
	\and 
	Henning Thomas \footnotemark[1]
}

\begin{document}
\maketitle

\begin{abstract}
A celebrated result of R\"odl and Ruci\'nski states that for every
graph \(F\), which is not a forest of stars and paths of length $3$, and fixed number of colours \(r\ge 2\) there exist positive constants \(c, C\) such that for \(p \leq cn^{-1/m_2(F)}\) the probability that
every colouring of the edges of the random graph \(G(n,p)\) contains a monochromatic copy
of \(F\) is \(o(1)\) (the ``0-statement''), while for \(p \geq Cn^{-1/m_2(F)}\) it is
\(1-o(1)\) (the ``1-statement''). Here $m_2(F)$ denotes the $2$-density of $F$. On the other hand, the case where $F$ is a forest of stars has a coarse threshold which is determined by the appearance of a certain small subgraph in $G(n, p)$.

Recently, the natural extension of the 1-statement of this theorem to \(k\)-uniform hypergraphs was proved by Conlon and Gowers and, independently, by Friedgut, R\"odl and Schacht. In particular, they showed an upper bound of order $n^{-1/m_k(F)}$ for the $1$-statement, where $m_k(F)$ denotes the $k$-density of $F$. Similarly as in the graph case, it is known that the threshold for star-like hypergraphs  is given by the appearance of small subgraphs. In this paper we show that another type of thresholds exists if $k \ge 4$: there are $k$-uniform hypergraphs for which the threshold is determined by the asymmetric Ramsey problem in which a different hypergraph has to be avoided in each colour-class. 

Along the way we obtain a general bound on the $1$-statement for asymmetric Ramsey properties in random hypergraphs. This extends the work of Kohayakawa and Kreuter, and of Kohayakawa, Schacht and Sp\"ohel who showed a similar result in the graph case. 
We prove the corresponding 0-statement for hypergraphs satisfying certain balancedness  conditions.
\end{abstract}

\section{Introduction}
\label{sec:intro}

Given graphs (or hypergraphs) $G$ and $F$, we denote by
$$
  G \rightarrow (F)_r
$$
the property that every colouring of the edges of
\(G\) with \(r\) colours contains a monochromatic copy of \(F\). If $r = 2$ we simply write $G \rightarrow (F)$. For example, Ramsey's theorem shows that for every two integers $\ell, r \ge 2$  there exists an integer $R(\ell,r)$ such that
$K_{R(\ell,r)} \rightarrow (K_\ell)_r$. In other words, every $r$-edge-colouring of a sufficiently large complete graph contains a monochromatic complete graph with $\ell$ vertices. In this paper we are interested in the case where $G$ is a \emph{binomial random graph}. 

A random graph \(G(n,p)\) is a graph on \(n\) vertices and each possible
edge is present with probability \(p\), independently of all other edges. 
The study of Ramsey-type questions in random graphs was initiated by \L{}uczak, Ruci\'nski and
Voigt  \cite{luczak1992ramsey} where, among other results, they established the 
threshold for 
$$
G(n,p) \rightarrow (K_3).
$$ 
In a subsequent series of
papers R\"odl and Ruci\'nski \cite{rodl1993lower,rodl1994random,rodl1995threshold}
fully solved the edge colouring problem (up to one corner case, noticed later
by Friedgut and Krivelevich \cite{FrKr00}). To state their results we first need
the following definition, which we give in the more general form for
\(k\)-uniform hypergraphs. For any graph or hypergraph \(G\) we denote by
\(e(G)\) and \(v(G)\) the number of its edges and vertices, respectively. 
Set
\begin{equation}
  \label{eq:def-m-k}
  d_k(G) :=
  \begin{cases}
    0 & \text{if \(e(G) = 0\)}\\
    1/k & \text{if \(e(G) = 1, v(G) = k\)} \\
    \frac{e(G)-1}{v(G)-k} & \text{otherwise},
  \end{cases}
\end{equation}
and
$$
  m_k(G) := \max_{H\subseteq G} d_k(H).
$$
We refer to $m_k(G)$ as the \emph{\(k\)-density} of \(G\). If \(d_k(G) = m_k(G)\) we say that $G$ is \emph{\(k\)-balanced}, and it is \emph{strictly
  \(k\)-balanced} if all strict subgraphs of \(G\) have smaller 
\(k\)-density. With $\Delta(G)$ we denote the maximum vertex degree of~$G$. 

\begin{theorem}[\cite{FrKr00,rodl1993lower,rodl1995threshold}]\label{thm:graph-RR}
  Let \(F\) be a graph with at least one edge and \(r \geq 2\).
  \begin{enumerate}[(i)]
  \item If \(F\) is a forest of stars, then
    \begin{equation*}
      \lim_{n\to\infty} \P[G(n,p) \rightarrow (F)_r] =
      \begin{cases}
        0 & \text{if \(p \ll n^{-1-1/(r(\Delta(F)-1)+1)}\)}\\
        1 & \text{if \(p \gg n^{-1-1/(r(\Delta(F)-1)+1)}\)}.
       \end{cases}
    \end{equation*}
  \item If \(r = 2\) and \(F\) is a forest of stars and at least one path with exactly 3
    edges, then there exists a constant \(C\) such that
    \begin{equation*}
      \lim_{n\to\infty} \P[G(n,p) \rightarrow (F)] =
      \begin{cases}
        0 & \text{if \(p \ll n^{-1/m_2(F)} = n^{-1}\)}\\
        1 & \text{if \(p \geq Cn^{-1/m_2(F)} = Cn^{-1}\)}.
      \end{cases}
    \end{equation*}
  \item \label{RRthm:iii}
   In all other cases there exist constants \(c = c(F,r)\) and \(C =
    C(F,r)\) such that
    \begin{equation*}
      \lim_{n\to\infty} \P[G(n,p) \rightarrow (F)_r] =
    \begin{cases}
        0 & \text{if \(p \leq cn^{-1/m_2(F)}\)}\\
        1 & \text{if \(p \geq Cn^{-1/m_2(F)}\)}.
    \end{cases}
  \end{equation*}
  \end{enumerate}
\end{theorem}

We refer the reader to \cite{nenadov2016short} for a short proof of part (iii) of Theorem \ref{thm:graph-RR}. 
Theorem \ref{thm:graph-RR}~\eqref{RRthm:iii}  was further strengthened by Friedgut and Krivelevich~\cite{FrKr00} and by Friedgut, R\"odl, Ruci\'nski and Tetali \cite{friedgut2006sharp} as follows. Friedgut and Krivelevich~\cite{FrKr00} proved the existence of a \emph{sharp} threshold for all forests $F$ and any number of colours where~\eqref{RRthm:iii} of Theorem~\ref{thm:graph-RR} applies. Friedgut, R\"odl, Ruci\'nski and Tetali \cite{friedgut2006sharp} showed the existence of a sharp threshold in the two-colour case where $F = K_3$. The latter result was recently extended to a more general class of graphs by Schacht and Schulenburg \cite{schacht2016sharp} who built on the ideas of Friedgut, H\`an, Person and Schacht~\cite{FHPS16}.  

Note that the expected number of copies of a graph \(F\) in
\(G(n,p)\) is \(\Theta(n^{v(F)}p^{e(F)})\) and the expected number of edges
 is \(\Theta(n^2p)\). The result above thus essentially states that for a balanced graph~$F$ the
transition from the 0- to the 1-statement happens around the value of \(p\) for which
these two quantities are roughly equal. In other words, if the expected number
of copies of \(F\) per edge is smaller than some small constant \(c'\) then
colouring without monochromatic \(F\) is possible, while if this number is bigger than a large constant \(C'\) then a monochromatic \(F\) always appears.
This can be explained by the following intuition: 
if each copy of $F$ contains an edge which does not belong to any other copy of $F$ then by colouring all such edges with red and every other edge with blue clearly gives a colouring without a monochromatic copy of $F$. If, on the other hand, each edge is contained in many copies of \(F\) then these must overlap heavily and a monochromatic copy is unavoidable.

There are two exceptional cases in Theorem~\ref{thm:graph-RR}: stars and paths
of length 3 (i.e. paths with exactly $3$ edges). In the case of a star \(S_\ell\) with \(\ell\) edges it is easy to
see by the pigeonhole principle that 
\(S_{r(\ell-1)+1} \rightarrow (S_\ell)_r\) for any \(r \geq 2\). In other words, as soon as a
star on \(r(\ell-1)+1\) edges appears in \(G(n,p)\) it is no longer possible
to colour it with \(r\) colours without a monochromatic \(S_\ell\). The threshold
for this event is asymptotically smaller than \(n^{-1/m_2(S_\ell)}\). In the case of \(P_3\)  and two colours a similar phenomenon
occurs. Given any cycle \(C_\ell\) of length \(\ell \geq 3\) we obtain a
``sunshine graph'' \(S_\ell^{*}\) by appending one pending edge to each vertex of
\(C_\ell\). For any odd \(\ell \geq 5\) it holds that \(S_\ell^{*}
\rightarrow (P_3)\). From standard results it follows that whenever \(p =
cn^{-1}\) there is a small but constant probability that \(G(n,p)\) contains
such a sunshine graph. Accordingly the 0-statement in (ii) cannot
be of the same type as in (iii).

While the graph case was solved completely by R\"odl and Ruci\'nski in the 90's, 
the hypergraph case is still open. Here we consider the random hypergraph model analogue to $G(n, p)$: $\Hknp$ is a hypergraph on $n$ vertices and each possible hyperedge with $k$ vertices is present with probability $p$, independently of all other hyperedges. R\"odl and Ruci\'nski \cite{rodl1998ramsey}
conjectured that the same intuition should hold as for Ramsey properties in the graph case, namely that a monochromatic
copy of \(F\) appears in every colouring whenever the expected number of copies
of \(F\) per hyperedge exceeds a large constant. They proved this for the complete 3-uniform
hypergraph on 4 vertices and 2 colours and, together with Schacht, extended it in \cite{rodl2007ramsey} to $k$-partite $k$-uniform
hypergraphs. Recently Friedgut, R\"{o}dl and Schacht~\cite{friedgut2010ramsey} proved the conjectured 1-statement  for all 
\(k\)-uniform hypergraphs. Similar results were obtained independently by Conlon and Gowers~\cite{conlon2013combinatorial}.

\begin{theorem}[\cite{conlon2013combinatorial,friedgut2010ramsey}]\label{thm:hypergraph-1-statement}
  Let \(F\) be a \(k\)-uniform hypergraph with maximum degree at least 2 and
  let \(r \geq 2\). There exists a constant \(C > 0\) such that for \(p =
  p(n)\) satisfying \(p \geq C n^{-1/m_k(F)}\) we have
  \begin{equation*}
    \lim_{n\to \infty} \P[\Hknp \rightarrow (F)_r] = 1.
  \end{equation*}
\end{theorem}

It was shown in \cite{nenadov2015algorithmic} that $n^{-1/m_k(F)}$ is indeed the threshold in the case where $F$ is a complete hypergraph and some further special classes were considered in \cite{thomas2013}. However, the complete characterisation, like the one in Theorem \ref{thm:graph-RR}, is still not known. As an evidence that such a characterisation might not be as simple as in the graph case, we show that there exists another type of thresholds if $k \ge 4$. In particular, contrary to the graph case we show that there exist hypergraphs for which the threshold is neither the conjectured $n^{-1/m_k(F)}$ nor is it determined by the appearance of a small subgraph. The following theorem is our first main contribution.

\begin{theorem} \label{thm:main_beta}
  For every $k \ge 4$ there exists a $k$-uniform hypergraph $F$ and positive constants $1 < \theta < m_k(F)$ and $c, C > 0$ such that
  $$
    \lim_{n\to\infty} \P[\Hknp \rightarrow (F)] =
      \begin{cases}
        0 & \text{if \(p \leq cn^{-1/\theta}\)}\\
        1 & \text{if \(p \geq Cn^{-1/\theta}\)}.
      \end{cases}    
  $$
\end{theorem}

Note that if the threshold for a hypergraph $F$ given by the previous theorem was determined by the appearance of a certain small hypergraph then, for the $0$-statement to hold, we would necessarily need to have $p \ll n^{-1/\theta}$ (we elaborate more on this in Section \ref{sec:main-proof}). The value of $\theta$ in Theorem~\ref{thm:main_beta} corresponds to the threshold for the \emph{asymmetric Ramsey} property, which we describe next. In Section \ref{sec:proof_outline} we then give a brief overview of the construction of $F$ and make the connection to the asymmetric case.

\subsection{Asymmetric Ramsey properties}

Instead of avoiding a monochromatic copy of the same hypergraph \(F\) in all
colours, in the asymmetric case we want to avoid a hypergraph
\(F_1\) in red, a hypergraph \(F_2\) in blue, and so on for all \(r \geq 2\)
colours. Similarly as before, let 
$$
  G \rightarrow (F_1, \ldots, F_r)
$$
denote the property that every
colouring of the edges of \(G\) with $r$ colours contains at least one monochromatic copy of
\(F_i\) in its respective colour (for some $1 \le i \le r$). Clearly, if all \(F_i\) are equal this reduces to the previously discussed (symmetric) case.

In the context of random graphs, this problem was first studied by Kohayakawa and Kreuter \cite{kohayakawa1997threshold} where they determined the threshold for the
case where each $F_i$ is a cycle. They
also conjectured that in the general case the threshold is determined by the
function below. 
Here we state the extension for \(k\)-uniform hypergraphs, while the original conjecture concerns only the case \(k
= 2\).
\begin{definition}\label{def:mk-asymm}
  Let \(F_1\), \(F_2\) be two \(k\)-uniform hypergraphs with at least one
  edge and such that \(m_k(F_1) \geq m_k(F_2)\). The \emph{asymmetric $k$-density} is defined as follows,
  \begin{equation}
    \label{eq:def-of-mk-asymm}
    m_k(F_1, F_2) = \max \left\{ \frac{e(F_1')}{v(F_1') - k + 1/m_k(F_2)} \; : \;
    F_1'\subseteq F_1, e(F_1') \geq 1 \right\}.
  \end{equation}
\end{definition}

Note that if \(m_k(F_1) = m_k(F_2)\) then \(m_k(F_1, F_2) = m_k(F_1)\) and otherwise $m_k(F_2) < m_k(F_1, F_2) < m_k(F_1)$. We say that
\(F_1\) is \emph{strictly balanced with respect to \(m_k(\cdot, F_2)\)} if no
proper subgraph \(F_1' \subsetneq F_1\) with at least one edge maximizes
\eqref{eq:def-of-mk-asymm}.

The intuition behind the conjectured value $n^{-1 / m_k(F_1, F_2)}$ for the asymmetric Ramsey property is easiest explained in the case $r = 3$ and $F_2 = F_3$. In other words, we have three colours and we aim to avoid a copy of $F_1$ in colour $1$ and a copy of $F_2$ in colours $2$ and $3$. First, observe that we can assign the colour $1$ to every edge which does not belong to a copy of $F_1$. Since $m_k(F_1, F_2) < m_k(F_1)$, for $p = \Theta(n^{-1/m_k(F_1, F_2)})$ we do not expect the copies of $F_1$ to overlap too much. Therefore, the number of edges which are left is of the same order as the number of copies of $F_1$, which is asymptotically $n^{v(F_1)} p^{e(F_1)}$. Assuming that these edges are randomly distributed (which is not entirely correct, but it gives a good intuition) this gives us a random hypergraph $H'$ with edge probability $p^* = n^{v(F_1) - k} p^{e(F_1)}$. Next, we use colours $2$ and $3$ for the hyperedges in $H'$. Now the same argument as in the symmetric case  applies: if the copies of $F_2$ are not overlapping heavily in $H'$ then it should be possible to assign two colours to the edges of $H'$ such that there is no monochromatic copy of $F_2$, and otherwise this is unavoidable. The reasoning as before shows that we expect this transition to happen around $p^* = n^{-1/m_k(F_2)}$. Putting all together, we obtain the value of $p$ given by the conjecture. 

In turns out that, in order to avoid a monochromatic $F_i$, if $p < cn^{-1/m_k(F_1, F_2)}$ the third colour is actually not needed. That is, we can assign colours $1$ and $2$ to $H'$ such that both monochromatic $F_1$ and $F_2$ are avoided. This is the reason why the conjectured threshold is determined only by the two graphs with largest \(k\)-density. Progress towards proving the
conjecture in the graph case was made by Marciniszyn et al. \cite{marciniszyn2009asymmetric}, where they confirmed it for complete graphs. They
also observed that the approach of Kohayakawa and Kreuter can be used to deduce the 1-statement for all graphs $F_1$ and $F_2$ which satisfy certain mild conditions, provided that the K\L R-conjecture holds (the K\L{}R-conjecture was verified recently by several groups of authors \cite{balogh2015independent,conlon2013klr,saxton2015hypergraph}). On the other hand, Kohayakawa, Schacht and Sp\"ohel \cite{kohayakawa2013upper} gave an alternative proof for the same result by using elementary means, with similar conditions on $F_1$ and $F_2$. Our second main contribution is an extension of these results to hypergraphs.

\begin{theorem} \label{thm:asymmetric-1}
Let $r \ge 2$ and $F_1, \ldots, F_r$ be $k$-uniform hypergraphs such that $m_k(F_1) \geq m_k(F_2) \geq \dotsb \geq m_k(F_r)$ and $F_1$ is strictly balanced with respect to $m_k(\cdot, F_2)$. Then there exists a constant $C > 0$ such that for $p \ge Cn^{-1/m_k(F_1, F_2)}$ we have
$$
  \lim_{n \rightarrow \infty} \Pr\left[ \Hknp \rightarrow (F_1, \ldots, F_r) \right] = 1.
$$
\end{theorem}

\subsubsection{Sufficient criterion for the $0$-statement}

The corresponding $0$-statement for the asymmetric Ramsey properties remains open in its full generality. 
As mentioned earlier, even in the case $k = 2$ it is known only for some special classes of graphs, such as complete graphs and cycles. 
 Our modest contribution towards resolving these questions is a result that reduces the problem from random graph theory to a deterministic question, at least under a certain balancedness condition. To state it we first need a couple of definitions.

\begin{definition} \label{def:asymmetric:balanced}
For given $k$-uniform hypergraphs $F_1$ and $F_2$, let \(\mathcal{F}^*(F_1, F_2)\) be the family defined in the following way,
\[
\mathcal{F}^*(F_1, F_2) := 
\left \{ F^* \; : \;  \parbox{0.65\displaywidth}{$\exists  F_2^* \cong F_2, e_0\in E(F_2^*)\text{ and }\exists F_1^e \cong F_1$ for all
$e \in E(F_2^*) - e_0$ s.t. \\ 
$e \in E(F_1^e)$ for all $e \in E(F_2^*) - e_0$ and  $F^* =F_2^* \cup   \bigcup_{e \in E(F_2^*) - e_0} F_1^e$}
\right\}.
\]
Note that hypergraphs $F_{1}^e$ need not be distinct.
\end{definition}

Informally, every graph in the family $\mathcal{F}^*(F_1, F_2)$ is obtained by an amalgamation of copies of $F_1$ onto $F_2$ which cover all but at most one edge of $F_2$. 

Given a hypergraph $F^* \in \cF^*(F_1, F_2)$, we say that $e_0$ is an \emph{attachment edge} of $F^*$. Note that there can be more than one edge which fall under the definition of an attachment edge. Moreover, we say that a  member $F^*\in\mathcal{F}^*(F_1, F_2)$  is {\em generic}  
if the following holds: each $F_1^e$ intersects $F_2^*$ on exactly $k$ vertices (namely those which correspond to the intersecting edge $e$) and the remaining vertices of $F_1^e$ are disjoint from those of all other $F_1^{e'}$. 
Observe that there could be several nonisomorphic generic copies 
 since an attachment edge can vary, and $F_1$ and $F_2$ need not be `symmetric'.

The main property that we require $\mathcal{F}^*$ to possess resembles that of strictly $k$-balancedness.

\begin{definition}[Asymmetric-balanced]
For given $k$-uniform hypergraphs $F_1$ and $F_2$, we say that $\mathcal{F}^*(F_1, F_2)$ is \emph{asymmetric-balanced} if the following two conditions are satisfied for all $F^* \in \mathcal{F}^*$ and all $H \subsetneq F^*$ with $V(H) \subsetneq V(F^*)$ that contain an attachment edge of $F^*$:
\begin{enumerate}
\item 	$$ \frac{e(F^*) - e(H)}{v(F^*) - v(H)} \geq m_k(F_1, F_2), $$
\item 
if 
$$\frac{e(F^*) - e(H)}{v(F^*) - v(H)} = m_k(F_1, F_2)$$
then $H$ consists of exactly an attachment edge and $F^*$ is generic.
\end{enumerate}
\end{definition}

The next theorem shows that the function $m_k(\cdot, \cdot)$ indeed determines the threshold for the asymmetric Ramsey property for all $k$-uniform hypergraphs which satisfy certain conditions. 

\begin{theorem} \label{thm:asymmetric-0}
Let $F_1, \ldots, F_r$ be $k$-uniform hypergraphs such that $F_2$ has at least three edges, $m_k(F_1) \geq m_k(F_2) \geq \dotsb \geq m_k(F_r)$ and the following holds:
\begin{enumerate}[(i)]
\item $F_1$ and $F_2$ are strictly balanced and $m_k(F_2) \ge 1$,
\item $F_1$ is strictly balanced with respect to $m_k(\cdot, F_2)$, 
\item $\mathcal{F}^*(F_1, F_2)$ is asymmetric-balanced,
\item for every $k$-uniform hypergraph $G$ such that $m(G) \leq m_k(F_1, F_2)$ we have
$$G \nrightarrow (F_1, F_2).$$
\end{enumerate}
Then there exists a constant $c > 0$ such that for $p \le cn^{-1/m_k(F_1, F_2)}$ we have
$$
\lim_{n \rightarrow \infty} \Pr \left[ \Hknp \rightarrow (F_1, \ldots, F_r) \right] = 0.
$$
\end{theorem}

We prove Theorem~\ref{thm:asymmetric-0} in Section \ref{sec:asymmetric}. In Section~\ref{sec:specialcase} we use it to derive the $0$-statement for pairs of  hypergraphs that are used in the proof of Theorem~\ref{thm:main_beta}.

\subsection{Asymmetric meets symmetric Ramsey} \label{sec:proof_outline}

Before we dive into proofs, we elaborate on the connection between the asymmetric Ramsey properties and Theorem \ref{thm:main_beta}.
We construct the hypergraph $F$ in Theorem \ref{thm:main_beta} as a disjoint union of two carefully chosen hypergraphs $F_1$ and $F_2$ with $m_k(F_1) > m_k(F_2)$. Moreover, we will choose $F_1$ as a star-like hypergraph whose threshold is asymptotically below $n^{-1/m_k(F_1, F_2)}$ (in particular, it is determined by the appearance of a small subgraph). The choice of $\theta := m_k(F_1, F_2)$ now comes into play as follows. First observe that if $G \not \rightarrow (F_1,F_2)$ then also $G \not \rightarrow (F_1\cup F_2)$. The $0$-statement thus follows immediately from the corresponding statement for the asymmetric case. For the $1$-statement we proceed as follows. 
Consider some $2$-edge-colouring of $H \sim \Hknp$. We arbitrarily partition the vertex set of $H$ into three parts of size $n/3$ and only consider the three induced (coloured) hypergraphs $H_1$, $H_2$ and $H_3$. 
By Theorem \ref{thm:asymmetric-1}, if $p > Cn^{-1/\theta}$ we know that $H_1$ either contains a blue $F_1$ or a red $F_2$ (or both). 
Similarly, by reverting the colours, Theorem \ref{thm:asymmetric-1} also implies that $H_2$  contains a red $F_1$ or a blue $F_2$. 
If in this way we find a red $F_1$ and a red $F_2$  or a blue $F_1$ and a blue $F_2$, we are done,
so we just have to consider the remaining two cases.

\begin{itemize}
  \item There exists a blue $F_1$ and a red $F_1$:

  \noindent
 As $\theta = m_k(F_1, F_2) > m_k(F_2)$ it follows from Theorem \ref{thm:hypergraph-1-statement} that there exists a monochromatic copy of $F_2$ in $H_3$, which, regardless of its colour, gives a monochromatic copy of $F$.

  \item There exists a red $F_2$ and a blue $F_2$:

  \noindent
  This is the case where the special choice of $F_1$ comes into play: we will choose it as a hypergraph for which the (symmetric) threshold is much lower than $n^{-1/m_k(F_1)}$, in fact, lower than $n^{-1/\theta}$. In particular 
there exists  $F'$ such that $F'\rightarrow (F_1)$ and $\Hknp$ a.a.s (asymptotically almost surely, i.e. with probability which tends to $1$ as $n$ goes to infinity) contains $F'$ for $p=n^{-1/m_k(F_1, F_2)}$.
   Thus, we conclude that $H_3$ contains $F'$ and, in turn, a monochromatic copy of $F_1$ which again implies the existence of a monochromatic $F$.
\end{itemize}

Of course, we need to show that it is possible to choose $F_1$ and $F_2$ with the desired properties. We are able to do so for $k \ge 4$, thus the bound in Theorem \ref{thm:main_beta}. Another challenge is to show that for $p < cn^{-1/\theta}$ we can colour the edges of $\Hknp$ such that there is no  red copy of $F_1$ and no blue copy of $F_2$, which implies that there is no monochromatic $F$.

The rest of the paper is organised as follows. 
 In the next section we describe a $k$-uniform hypergraph $F = F_1 \cup F_2$ and, 
 assuming that the threshold for $\Hknp \rightarrow (F_1, F_2)$ is $n^{-1/m_k(F_1, F_2)}$, prove Theorem \ref{thm:main_beta}. In Section \ref{sec:asymm_proof} we prove Theorem \ref{thm:asymmetric-1} (the $1$-statement for asymmetric Ramsey properties). In the rest of the paper we prove the matching lower-bound on the threshold for $\Hknp \rightarrow (F_1, F_2)$. In Section \ref{sec:asymmetric} we prove Theorem \ref{thm:asymmetric-0} 
 and then in Section \ref{sec:specialcase} we verify that it can be applied with $F_1$ and $F_2$. We close with some concluding remarks in Section~\ref{sec:conclusion}.

\section{Proof of Theorem \ref{thm:main_beta}}
\label{sec:main-proof}

Let \(G = (V, E)\) be a graph and \(W\) a set of \(k-2\) additional vertices
with \(V \cap W = \emptyset\). We denote by \(G^{+k} = (V', E')\) the
\(k\)-uniform hypergraph obtained by adding the vertices of \(W\) to each edge
of \(G\), i.e.\ we set \(V' = V \cup W\) and \(E' = \{ e \cup W \mid e \in
E\}\).  A \emph{tight} $k$-cycle $C_t$ is a $k$-uniform hypergraph with vertex set $\{v_0, \ldots, v_{t-1}\}$ and the edges $\{v_i, v_{i+1}, \ldots, v_{i + k - 1} \}$ for every $0 \le i \le t - 1$ (the index addition is modulo $t$).

The following theorem implies Theorem \ref{thm:main_beta}.

\begin{theorem} \label{thm:main_aux}
  For every $k \ge 4$ there exist positive constants $c, C$ such that
  $$
    \lim_{n\to\infty} \P[\Hknp \rightarrow (\HFi \cup \HFii)_r] =
      \begin{cases}
        0 & \text{if \(p \leq cn^{-1/\theta}\)}\\
        1 & \text{if \(p \geq Cn^{-1/\theta}\)},
      \end{cases}    
  $$
  where
  $$
    \theta := m_k(K_3^{+k}, \HFii)
  $$
  and
  $$
    t_4 = 8, \; t_5 = 14, \text{ and } \quad t_k = k^2 \; \text{ for } \; k \ge 6.
  $$
\end{theorem}
From the definition of the $k$-density $m_k$ we have
$$m_k(\HFii) = \frac{t_k-1}{t_k - k} \text{ \ \ and \ \ } m_k(K^{+k}_3) = 2$$
and thus $m_k(\HFii) < m_k(K_3^{+k})$ for our choice of $t_k$. In addition, one easily checks that $K_3^{+k}$ is strictly balanced with respect to $m_k(\cdot,\HFii)$. 
Thus,
\[
\theta = m_k(K^{+k}_3, \HFii)=\frac{3t_k-3}{2t_k-k-1} < m_k(\HFii\cup K^{+k}_3)=2.
\]

Observe that the threshold $\theta$ from Theorem~\ref{thm:main_aux} 
does not fall into any category (i)--(iii) from Theorem \ref{thm:graph-RR}. Indeed, the fact that
$\theta < m_k(K_3^{+k} \cup \HFii)$ excludes (ii) and (iii). We further exclude the possibility that there exists a small hypergraph whose appearance determines the threshold. To see this let \(m(F)\) denote the usual density measure
\begin{equation*}
  m(F) := \max_{\substack{H\subseteq F\\v(H) \geq 1}} \frac{e(H)}{v(H)}.  
\end{equation*}
and note that Bollob\'as' small subgraphs theorem
\cite{bollobas1981threshold} (for graphs) extends straightforwardly to hypergraphs. That is,
for any hypergraph \(F\) we have
\begin{equation}
  \label{eq:small-subgraphs-hypergraph}
  \lim_{n\to\infty} \P[\text{\(\Hknp\) contains a copy of \(F\)}] =
  \begin{cases}
    0 & \text{if \(p \ll n^{-1/m(F)}\)}\\
    1 & \text{if \(p \gg n^{-1/m(F)}\)}.
  \end{cases}
\end{equation}
Therefore, if the threshold in Theorem \ref{thm:main_aux} is determined by the appearance of a small subgraph, as in the $0$-statement of (i) in Theorem~\ref{thm:graph-RR},  then we would necessarily have $p \ll n^{-1/\theta}$ in order for the $0$-statement to hold, which yields a longer interval for the phase transition, contrary to what we proved.

The proof of Theorem \ref{thm:main_aux} relies on Theorem \ref{thm:asymmetric-1} (which is proven in the next section) and the following lower bound on the threshold for $\Hknp \rightarrow (K_3^{+k}, \HFii)$ which we prove in Section~\ref{sec:asymmetric}.

\begin{lemma} \label{lemma:0_statement_aux}
  For every $k \ge 4$ there exists $c > 0$ such that if $p < cn^{-1/m_k(K_3^{+k}, \HFii)}$ then
  $$
    \lim_{n\to\infty} \P[\Hknp \rightarrow (K_3^{+k}, \HFii)] = o(1).
  $$
\end{lemma}

The choice of $t_k$ in Theorem~\ref{thm:main_aux} is based on a number of inequalities that have to be satisfied. Some of them come from the proof of Lemma \ref{lemma:0_statement_aux}, some others will become clear soon. In particular, the reason why we need $k\ge 4$  is that for small values of $k$ the calculations behave differently. 

\begin{proof}[Proof of Theorem \ref{thm:main_aux}]
Note that if $\Hknp \not \rightarrow (K^{+k}_3, \HFii)$ then also $\Hknp \not \rightarrow (K^{+k}_3\cup \HFii)$. Therefore, Lemma~\ref{lemma:0_statement_aux} immediately gives the $0$-statement.

For the $1$-statement we proceed as explained in Section~\ref{sec:proof_outline}. 
Recall from there that the only fact that we have to check is that 
$$
\lim_{n\to\infty} \P[\Hknp \rightarrow (K^{+k}_3)] = 1\qquad \text{for \(p \geq Cn^{-1/\theta}\)}.
$$
Now we use that the density of $K^{+k}_6$ is $m(K^{+k}_6) = 15/(k+4) < \theta$ (this inequality is the main reason why we require $k\ge 4$). As $K_6\rightarrow (K_3)$ and thus
$K^{+k}_6\rightarrow (K^{+k}_3)$ we know that $H_3$ contains a monochromatic copy of $K^{+k}_3$, which concludes the proof of the theorem.
\end{proof}

\section{Asymmetric Ramsey properties: 1-statement} \label{sec:asymm_proof}

We generalize an approach of Nenadov and Steger \cite{nenadov2016short} based on the hypergraph containers. For further applications of this method in the context of Ramsey-type problems we refer the reader to \cite{conlon2016note,nenadov2015threshold,rodl2015exponential}. The proof relies on three ingredients: Ramsey's theorem, Janson's inequality and hypergraph containers. We now state each of them.

The following theorem is a well known quantitative strengthening of Ramsey's theorem. We include the proof for convenience of the reader.

\begin{theorem}[folklore] \label{thm:ramsey_count}
Let $F_1, \ldots, F_r$ be $k$-uniform hypergraphs and $r \in \NN$ be a constant. Then there exist constants $\alpha > 0$ and $n_0 \in \NN$ such that for any $n \ge n_0$ and any $r$-edge-colouring of $K_n^{(k)}$ (the complete $k$-uniform hypergraph on $n$ vertices) there are at least $\alpha n^{v(F_i)}$ copies of $F_i$ in the colour $i$, for some $1 \le i \le r$.
\end{theorem}
\begin{proof}
From Ramsey's theorem we know that there exists $N = N(F_1, \ldots, F_r) \in \NN$ such that every
$r$-edge-colouring of the edges of $K_N^{(k)}$ contains a copy of $F_i$ in colour $i$, for some $i \in [r]$. Therefore, in any $r$-edge-colouring of $K_n^{(k)}$ with $n \ge N$, every $N$-subset of the vertices contains at least one monochromatic copy of some $F_i$ in colour $i$. In particular, there exists $i \in [r]$ such that in at least $\binom{n}{N}/r$ $N$-subsets of $K_n^{(k)}$ we find a copy of $F_i$ in colour $i$. On the other hand, every copy of $F_i$ is contained in at most $\binom{n - v(F_i)}{N - v(F_i)}$ $N$-subsets thus the number of different monochromatic copies of $F_i$ is at least
$$
  \binom{n}{N} \left(r \binom{n-v(F_i)}{N-v(F_i)} \right)^{-1} \ge \frac{\left(n / N\right)^N}{r n^{N - v(F_i)}} \ge \frac{n^{v(F_i)}}{r N^N}.
$$
The theorem now follows for $\alpha = (r N^{N})^{-1}$.
\end{proof}

Next, we derive a bound on the expected number of copies of certain hypergraphs in $\Hknp$.
\begin{lemma}
\label{lemma:funionf}
Let $F_1$ and $F_2$ be $k$-uniform hypergraphs such that $m_k(F_1)\ge m_k(F_2)>0$ and  $F_1$ is strictly balanced with respect to $m_k(\cdot, F_2)$. 
Let $\mathcal F$ be a family of subgraphs of $K_n^{(k)}$  such that  each member $F \in \mathcal F$ is a union of two distinct $F_1$-copies intersecting on at least one edge.
For a positive constant $C$, $p = C n^{-1/m_k(F_1, F_2)}$ and  $H \in \Hknp$, let $X$ be the number of hypergraphs $F \in \mathcal F$ contained in $H$. Then there exist a constant $\delta > 0$ such that
$$
E[X]  \le C' n^{k - 1/m_k(F_2) - \delta/2}.
$$
\end{lemma}
\begin{proof}
Let $F \in \mathcal F$ be an arbitrary member of $\cF$. Let  $F', F'' \subset K_n^{(k)}$ be two $F_1$-copies 
such that $F = F' \cup F''$ and $e(F' \cap F'') \ge 1$. Set $S := F' \cap F''$.
From the assumption that $F_1$ is strictly balanced with respect to $m_k(\cdot, F_2)$ and $S$ is isomorphic to a proper subgraph of $F_1$ we have 
\begin{equation} \label{eq:delta_S}
  \frac{e(S)}{v(S) - k + 1/m_k(F_2)} < m_k(F_1, F_2) - \delta_S
\end{equation}
for some $\delta_S > 0$. 
Since there are only constantly many subgraphs $S \subseteq F_1$, there exists a constant $\delta' > 0$ such that the previous inequality holds with $\delta_S = \delta'$ for every $S \subseteq F_1$ with $e(S) \ge 1$. Using the assumption $p = Cn^{-1/m_k(F_1, F_2)}$, a straightforward calculation shows that the expected number of subgraphs of $H$ isomorphic to $F' \cup F''$ is at most of order
\[
  n^{v(F' \cup F'')} p^{e(F' \cup F'')} = 
  n^{2v(F_1) - v(S)} p^{2e(F_1) - e(S)} = 
  \frac{O(n^{2(k - 1/m_k(F_2))})}{n^{v(S)}p^{e(S)}} 
   \stackrel{\eqref{eq:delta_S}}{=} O(n^{k - 1/m_k(F_2) - \delta/2}),
\]
where $\delta>0$ depends only on $\delta'$, $F_1$ and $F_2$. Since there are only constantly many ways to obtain a graph as a union of two copies of $F_1$, this concludes the proof.
\end{proof}

We also need the following statement on the probability of existence of certain hypergraphs in the random hypergraph $\Hknp$. As the proof follows almost directly from Janson's inequality, we give just a sketch of the argument.


\begin{lemma} \label{lemma:janson}
Let $F_1$ and $F_2$ be $k$-uniform hypergraphs such that $m_k(F_1)\ge m_k(F_2)>0$ and $F_1$ is strictly balanced with respect to $m_k(\cdot, F_2)$, and let $\varepsilon > 0$ be a constant. Then there exists a constant $\beta > 0$ such that the following holds for any constant $C > 0$: for $p = Cn^{-1/m_k(F_1, F_2)}$ and any family $\cF$ of subgraphs of $K_n^{(k)}$ isomorphic to $F_1$ with $|\cF| \ge \varepsilon n^{v(F_1)}$ we have 
$$
  \Pr \left[F' \nsubseteq \Hknp \text{ for all } F' \in \cF  \right] \le e^{-\beta C n^{k - 1/m_k(F_2)}}.
$$
\end{lemma}
\begin{proof}
Let $X$ be the number of members of $\mathcal F$ appearing in $\Hknp$. We bound $E[X]$ as follows,
$$
E[X] = |\mathcal F| p^{e(F)} \ge \varepsilon n^{v(F_1)} (Cn^{-1/m_k(F_1, F_2)})^{e(F_1)}
\ge \eps C n^{k - 1/m_k(F_2)}.
$$
By Lemma \ref{lemma:funionf} and  Janson's inequality  (e.g., see 
Theorem~$2.14$ in~\cite{purple})
  we have 
$$
\Pr[ X < (1 - \alpha) E[X]] \le e^{-\alpha^2 E[X] /3},
$$ for any constant $\alpha > 0$. Using the estimate on $E[X]$, this implies the lemma. 
\end{proof}

Finally, we state our main tool, the \emph{container} theorem of Saxton and Thomason \cite{saxton2015hypergraph} (the similar result was independently obtained by Balogh, Morris and Samotij \cite{balogh2015independent}). 

\begin{definition}
For a given set $S$ and $\ell \in \NN$, let $\cT_{\ell}(S)$ denote the family of $\ell$-tuples of subsets of $S$ (not necessarily disjoint), i.e.
$$ 
  \cT_{\ell}(S) = \left\{ T = (S_1, \ldots, S_\ell) \mid S_i \subseteq S \; \text{for} \; 1 \leq i \leq \ell  \right\}.
$$
\end{definition}

Given a set $S$ and an $\ell$-tuple $T = (S_1, \ldots, S_\ell) \in \cT_\ell(S)$, let $|T|$ denote the size of the union of $S_i$'s, i.e. $|T| := \left| \bigcup_{i \in [\ell]} S_i\right|$. Moreover, for a subset $S' \subseteq S$ we write $T \subseteq S'$ to denote that $S_i \subseteq S'$ for every $i \in [\ell]$.



\begin{theorem}[Theorem $2.3$ from~\cite{saxton2015hypergraph}] \label{thm:container}
Let $F$ be a $k$-uniform hypergraph with $e(F) \geq 2$ and let $\eps > 0$ be a constant. Then there exists $\ell \in \NN$ such that for every $n\ge \ell$ there exists a function $f : \cT_{\ell}(E(K_n^{(k)})) \to 2^{E(K_n^{(k)})}$ with the following property: for every $F$-free $k$-uniform hypergraph $H \subseteq K_n^{(k)}$ there exists $T \in \cT_{\ell}(E(K_n^{(k)}))$ such that 
\begin{enumerate}
\item[(a)] $T \subseteq E(H) \subseteq f(T)$,
\item[(b)] the number of edges in $T$ is at most $|T| \le \ell n^{k - 1/m_k(F)}$,
\item[(c)] the hypergraph induced by the edge set $f(T)$ contains at most $\eps n^{v(F)}$ copies of $F$.
\end{enumerate}
\end{theorem}

With these statements at hand, we are ready to prove Theorem \ref{thm:asymmetric-1}.

\begin{proof}[Proof of Theorem \ref{thm:asymmetric-1}]
Let $F_1, \ldots, F_r$ be $k$-uniform hypergraphs such that $m_k(F_1) \ge m_k(F_2) \ge \ldots \ge m_k(F_r)$ and $F_1$ is strictly balanced with respect to $m_k(\cdot, F_2)$. Since the property
$$
  H \rightarrow (F_1, \ldots, F_r)
$$
is monotone increasing, we may assume that $p = Cn^{-1/m_k(F_1, F_2)}$ for some constant $C > 0$ which we determine later. This assumption is not necessary, but it will make calculations easier. 

The overall proof strategy is as follows.  If a given $k$-uniform hypergraph $H$ on $n$ vertices is not Ramsey, i.e.\ 
 $ H \not\rightarrow (F_1, \ldots, F_r)$, then there exists a partition $E_1, \ldots, E_r \subseteq E(H)$ such that the $k$-uniform hypergraph $G_i = (V(H), E_i)$ is $F_i$-free for every $1 \le i \le r$. We clearly think of the edges from $E_i$ as being coloured in colour $i$. Next, we use Container Theorem, Theorem~\ref{thm:container}, to `place' each $E_i$ (where $2\le i\le r$) into some container $C_i$ with less than $\alpha n^{v(F_i)}$ copies of $F_i$. This way Theorem~\ref{thm:ramsey_count} asserts that the hypergraph $R$ on the remaining edges, i.e.\ $R=\left([n],\binom{[n]}{k}\setminus \cup_{i=2}^r E_i\right)$ contains at least $\alpha n^{v(F_1)}$ copies of $F_1$. Assuming that 
 $H\sim \Hknp$, we infer that the edges from $E(H)\cap E(R)$ will be colored in colour $1$, which will allow us to get sufficiently small probability that none of the at least $\alpha n^{v(F_1)}$ copies of $F_1$ lands in $E_1$ (this is an application of Janson's inequality, Lemma~\ref{lemma:janson}). Of course, there are some subtleties as to how we  define certain probability events. In particular the options for $C_i$'s and $R$ have to be considered `beforehand' (via Theorem~\ref{thm:container}).
 
For $H\sim\Hknp$ with $p$ as specified above, we expect that most copies of $F_1$ do not have edges in common (i.e.\ are isolated). This will be indeed the case with high probability and the proof outline will be carried out in this case.  The unlikely case will simply follow from Markov's inequality, which we first turn to. 

 \textbf{Many non-isolated copies.} 
Let $\cF=\cF(F_1,H)$ be the family of all subgraphs of $H$ isomorphic to $F_1$ and let $\cI = \cI(F_1,H) \subseteq \cF$ be the subfamily of all  \emph{isolated} subgraphs, i.e.
$$
  \cI = \left\{ F' \in \cF \mid \forall F'' \in \cF \setminus \{F'\}, \; E(F') \cap E(F'') = \emptyset \right\}.
$$
\begin{claim}\label{cl:few_nonisolated}
There exists a constant $\delta > 0$ such that for $H\sim\Hknp$ we a.a.s. have
\begin{equation} \label{eq:F_minus_I}
  |\cF \setminus \cI| \le n^{k - 1/m_k(F_2) - \delta}.
\end{equation}
\end{claim} 
\begin{proof}
By Lemma \ref{lemma:funionf} we know that the expected number of hypergraphs $F \subseteq H$ which can be obtained as a union of two distinct $F_1$-copies intersecting on at least one edge is at most $n^{k - 1/m_k(F_2) - \delta/2}$, for some constant $\delta > 0$.
 Note that for each $F' \in \cF \setminus \cI$ 
there exists $F'' \in \cF$ such that $S := F' \cap F''$ contains at least one edge
and  $F:=F' \cup F'' \subseteq H$. As  there are only 
constantly many different copies of $F_1$ contained in $F=F' \cup F''$,  by  previous  observations we have that the expected size of $|\cF \setminus \cI|$ is $O(n^{k - 1/m_k(F_2) - \delta/2})$. 
From Markov's inequality we obtain that the actual number of such subgraphs is a.a.s.\ at most $n^{k - 1/m_k(F_2) - \delta}$.
\end{proof}


 \textbf{Few non-isolated copies.}
Let us assume that $H$ is such that the bound in \eqref{eq:F_minus_I} holds. Next, note that for each $F' \in \cF$ at least one edge from $E(F')$ does not belong to $E_1$, as otherwise there exists a copy of $F_1$ in $G_1$, which is monochromatic. Moreover, we can assume that this holds for \emph{exactly} one edge if $F' \in \cI$: since each edge of $F'$ belongs to exactly one copy of $F_1$ in $H$ (follows from the definition of $\cI$), by re-colouring all but one (arbitrarily chosen) edge from $E(F') \setminus E_1$ with colour $1$ we do not create a copy of $F_1$ in colour $1$. Since no new edge gets a colour $i \ge 2$, this clearly does not create a monochromatic copy of any $F_i$ in the corresponding colour. Next, for $i \ge 2$ we partition each colour class $E_i$ into $I_i$ (``isolated'' edges in $E_i$) and $L_i$ (``leftover'' edges),
$$
  I_i = \bigcup_{F' \in \cI} E(F') \cap E_i \quad \text{and} \quad L_i = E_i \setminus I_i.
$$
By the previous assumption (that all but exactly one edge in $F' \in \cI$ belong to $E_1$) we have that for each edge $e \in I := \bigcup_{i = 2}^r I_i$ there exists a unique $F_e \in \cI$ with $e \in E(F_e)$, and $E(F_e) \cap E(F_{e'}) = \emptyset$ for different edges $e, e' \in I$. Finally, note that we can also assume that every edge in $H$ which does not belong to a copy of $F_1$ has  colour $1$. It then follows that each edge in $L_i$ belongs to some $F' \in \cF \setminus \cI$, and from $ |\cF \setminus \cI| \le n^{k - 1/m_k(F_2) - \delta}$ we conclude
\begin{equation} \label{eq:E_i}
  |L_i| = O(n^{k - 1/m_k(F_2) - \delta}).
\end{equation}

Next, we use the container theorem to ``approximate'' each of the sets $I_i$. Let $\alpha$ be the constant given by Theorem~\ref{thm:ramsey_count} for $F_1, \ldots, F_r$ and set $\varepsilon = \alpha / 2$. 
Furthermore, let $\ell_i \in \NN$ and $f_i : \cT_{\ell_i}(E(K_n^{(k)})) \to 2^{E(K_n^{(k)})}$ be functions obtained by applying Theorem \ref{thm:container} with $\varepsilon$ and $F = F_i$, for each $i \ge 2$.
 Since the hypergraph induced by the set of edges $I_i$ is $F_i$-free there exists an 
 $\ell_i$-tuple $T^i \in \cT_{\ell_i}(E(K_n^{(k)}))$ 
 such that 
$$
 T^i \subseteq I_i \subseteq f_i(T^i).
$$
Let
$$
  R = E(K_n^{(k)}) \setminus \bigcup_{i = 2}^r (f_i(T^i) \cup L_i).
$$
Note that set $R$ is uniquely determined by $\mathbf T = (T^2, \ldots, T^r)$ and 
 $\mathbf{L} = (L_2, \ldots, L_r)$, where each $L_i \subseteq E(K_n^{(k)})$ is a subset of 
size $O(n^{k - 1/m_k(F_2) - \delta})$. 
Therefore, we can enumerate all $R$ by going over all possible choices for $\mathbf T$ and $\mathbf L$. 
We refer to the set $R$ fixed by the choice of $\mathbf T$ and $\mathbf L$  as $R(\mathbf T, \mathbf L)$.
The following claim plays the central role in our argument.

\begin{claim} \label{claim:R}
Given the tuples $\mathbf T = (T^2, \ldots, T^r)$ and $\mathbf{L} = (L_2, \ldots, L_r)$ as described above, the hypergraph induced by the set of edges $R=R(\mathbf{T},\mathbf{L})$ contains at least $\alpha n^{v(F_1)}$ copies of $F_1$.
\end{claim}
\begin{proof}
  From Theorem \ref{thm:container} we have that each $f_i(T^i)$ ($i \ge 2$) contains at most $\alpha n^{v(F_i)} / 2$ copies of $F_i$. Furthermore, the number of copies of $F_i$ in $f_i(T^i) \cup L_i$ which contain an edge from $L_i$ is at most $n^{v(F_i) - k} |L_i|$ and from \eqref{eq:E_i} we conclude that there are $o(n^{v(F_i)})$ such copies. In total, $f_i(T^i) \cup L_i$ contains less than $\alpha n^{v(F_i)}$ copies of $F_i$, for every $i \ge 2$.

  Consider the auxiliary $r$-edge-colouring of $K_n^{(k)}$ defined as follows: an edge $e$ gets the colour $1$ if $e \in R$ and otherwise it gets an arbitrary colour $i \ge 2$ such that $e \in f_i(T^i) \cup L_i$. From the previous discussion we have that each colour $i \ge 2$ contains less than $\alpha n^{v(F_i)}$ copies of $F_i$ and Theorem \ref{thm:ramsey_count} implies that $R$ has to contain at least $\alpha n^{v(F_1)}$ copies of $F_1$.
\end{proof}

Since $E_i = I_i \cup L_i \subseteq f_i(T^i) \cup L_i$ and each $F' \in \cF$ contains an edge from some $E_i$ with $i \ge 2$ we conclude that $E(F') \not \subseteq R$ for every $F' \in \cF$. That is, the hypergraph $H$ completely avoids all copies of $F_1$ which are contained in $R$. Recall that for $\mathbf{T}=(T^2,\ldots, T^r)$  we have $T^i_j \subseteq I_i$ for all $j\in[\ell_i]$. We define the set $T = T(\mathbf{T})$ as follows:
$$
T:=T(\mathbf{T}) := \bigcup_{i \ge 2}\bigcup_{j=1}^{\ell_i} T^i_j.
$$
Observe that for each $e \in T$
there exists a copy of $F_1$ in $H$ containing $e$, say $F_e$, such that $E(F_e) \cap E(F_{e'}) = \emptyset$ for different $e, e' \in T$ (as we have already observed this for all edges in $I \supseteq T$). Thus, we obtain a set of copies of $F_1$ in $H$ which are `rooted' at the edges from $T$.

To summarize, if $H \nrightarrow (F_1, \ldots, F_r)$ then either $\cF \setminus \cI$ is bigger than $n^{k - 1/m_k(F_2) - \delta}$  or there exist $\mathbf T = (T_2, \dots, T_r)$ and $\mathbf{L} = (L_2, \ldots, L_r)$ such that $H$ satisfies properties $\cP_1(R)$ and $\cP_2(T)$, where  $R = R(\mathbf T, \mathbf L)$ and $T = T(\mathbf T)$ are as defined earlier and 
\begin{itemize}
  \item $\cP_1(R)$ denotes the property that $R \cap E(H)$ does not contain a copy of $F_1$, 
  \item $\cP_2(T)$ denotes the property that for each $e \in T$ there exists a copy of $F_1$ in $H$ (denoted by $F_e$) such that $E(F_e) \cap E(F_{e'}) = \emptyset$ for different $e, e' \in T$.
\end{itemize}

 \textbf{Estimating the probability $H \nrightarrow (F_1, \ldots, F_r)$.}
 Finally, for $H\sim\Hknp$, 
 we can upper bound the probability that $H \nrightarrow (F_1, \ldots, F_r)$ as follows
 (recall $\cF=\cF(F_1,H)$ and $\cI=\cI(F_1,H)$): 
\begin{align*}
  \Pr[ H \nrightarrow (F_1, \ldots, F_r) ] &\le \Pr[ |\cF \setminus \cI| > n^{k - 1/m_k(F_2) - \delta} ] + \Pr[\exists \mathbf T, \mathbf{L} \; : \; H \in \cP_1(R(\mathbf T,\mathbf{L})) \cap \cP_2(T(\mathbf T))] \\
    &\le o(1) + \sum_{\mathbf T} \sum_{\mathbf{L}} \Pr[ H \in \cP_1(R(\mathbf T,\mathbf{L})) \cap \cP_2(T(\mathbf T))].
\end{align*}
 The first probability is $o(1)$ by Claim~\ref{cl:few_nonisolated}. 
Note that $\cP_1(R)$ is a decreasing and $\cP_2(T)$ is an increasing graph property, thus by the FKG inequality (see e.g.\ Theorem~6.3.3 in~\cite{alon2015probabilistic}) we know they are negatively correlated, i.e.
\begin{equation} \label{eq:H_not_Ramsey}
  \Pr[ H \nrightarrow (F_1, \ldots, F_r) ] \le o(1) + \sum_{\mathbf T} \sum_{\mathbf{L}} \Pr[ H \in \cP_1(R(\mathbf T,\mathbf{L})) ] \cdot \Pr[ H \in \cP_2(T(\mathbf T))]
\end{equation}
Our aim is to show that the double sum also sums up to $o(1)$. 

For fixed choices of $\mathbf T$ and $\mathbf{L}$, and therefore for fixed $R$, from  Lemma \ref{lemma:janson}  and Claim \ref{claim:R}  we get
$$
  \Pr[\cP_1(R)] \le e^{-\beta Cn^{k - 1/m_k(F_2)}},
$$
for some $\beta > 0$. On the other hand, the expected number of mappings $T \rightarrow \cF$ which are witnesses for the property $\cP_2(T)$ is at most
$$
  (n^{v(F_1) - k} p^{e(F_1)})^{|T|} = C^{e(F_1)|T|} n^{-|T| / m_k(F_2)}.
$$
Consequently, the probability that $H$ has the property $\cP_2(T)$ is at most this value. We can now upper bound the sum in \eqref{eq:H_not_Ramsey} as
\begin{equation} \label{eq:H_not_Ramsey1}
  \Pr[ H \nrightarrow (F_1, \ldots, F_r) ] \le o(1) +  e^{-\beta Cn^{k - 1/m_k(F_2)}} \cdot\left(\sum_{\mathbf T} \sum_{\mathbf{L}}  C^{e(F_1)|T|} n^{-|T| / m_k(F_2)}\right).
\end{equation}
Next, from $m_k(F_2) \ge m_k(F_i)$ for $i \ge 3$ we observe that Theorem~\ref{thm:container} implies that $|T^i| \le \ell n^{k - 1/m_k(F_2)}$ for $i \ge 2$, where $\ell = \max_{i \ge 2} \ell_i$. Therefore, for each $t \in \NN$ with $t \le r \ell n^{k - 1/m_k(F_2)}$ (the maximal size of $T$) we can upper bound the number of 
choices of $\mathbf T$ by picking $t$ edges and for each $T^i = (T^i_1, \ldots, T^i_{\ell_i}) \in \cT_{\ell_i}$ and each $j \in \{1, \ldots, \ell_i\}$ we decide which edges go to $T^i_j$. We can do this in at most
$$
  \binom{n^k}{t} 2^{\ell r t}
$$
ways, where the $2^{\ell r t }$ factor comes from the upper bound on the number of ways how to distribute $t$
edges among $T_j^i$'s. On the other hand, we can choose each $L_i$ in at most
\begin{equation}\label{eq:balancedness}
  \binom{n^k}{C'n^{k - 1/m_k - \delta}} \le n^{k \cdot C' n^{k - 1/m_k(F_2) - \delta}} = e^{o(n^{k - 1/m_k(F_2)})}
\end{equation}
ways, for some constant $C' > 0$, cf.~\eqref{eq:E_i}. Using these estimates, we further bound the double sum in \eqref{eq:H_not_Ramsey1} as follows
$$
   \sum_{\mathbf T} \sum_{\mathbf{L}} \ldots
 \le e^{o(n^{k - 1/m_k(F_2)})} \sum_{t = 0}^{r \ell n^{k - 1/m_k(F_2)}}  \binom{n^k}{t} 2^{r \ell t} C^{e(F_1)t} n^{-t/m_k(F_2)} .
$$
Using the estimate $\binom{n}{\ell}\le (en/\ell)^\ell$, we further simplify the sum on the right hand side, 
\begin{align*}
  \sum_{t = 0}^{r \ell n^{k - 1/m_k(F_2)}} \binom{n^k}{t} 2^{r \ell t} C^{e(F_1)t} n^{-t/m_k(F_2)} 
  &\le \sum_{t = 0}^{r \ell n^{k - 1/m_k(F_2)}} \left( \frac{e n^k }{t} 2^{r \ell} C^{e(F_1)} n^{-1/m_k(F_2)} \right)^t \\
  &\le (r\ell n^{k-1/m_k(F_2)}+1) \left( \frac{e 2^{r\ell} C^{e(F_1)}}{r \ell} \right)^{r \ell n^{k - 1/m_k(F_2)}} \\
  &\le ( e 2^{r\ell} C^{e(F_1)})^{r \ell n^{k - 1/m_k(F_2)}},
\end{align*}
where we used the fact that the maximal element of the sum is attained for $t = r \ell n^{k - 1/m_k(F_2)}$, 
since $g(t):=(m/t)^t$ is monotone increasing for $x\in [0,m/e]$. Finally, we obtain using~\eqref{eq:H_not_Ramsey1}:
$$
  \Pr[ H \nrightarrow (F_1, \ldots, F_r) ] \le o(1) + e^{-\beta Cn^{k - 1/m_k(F_2)}} e^{o(n^{k - 1/m_k(F_2)})}  ( e 2^{r\ell} C^{e(F_1)} )^{r \ell n^{k - 1/m_k(F_2)}} = o(1),
$$
for sufficiently large $C > 0$.
\end{proof}

\section{Asymmetric Ramsey properties: 0-statement}
\label{sec:asymmetric}

In this section we prove Theorem \ref{thm:asymmetric-0}, which we repeat here for the convenience of the reader.

\begin{theorem*} 
Let $F_1, \ldots, F_r$ be $k$-uniform hypergraphs such that $F_2$ has at least three edges, $m_k(F_1) \geq m_k(F_2) \geq \dotsb \geq m_k(F_r)>0$ and
\begin{enumerate}[(i)]
\item $F_1$ and $F_2$ are strictly balanced and $m_2(F_2) \ge 1$,
\item $F_1$ is strictly balanced with respect to $m_k(\cdot, F_2)$,
\item $\mathcal{F}^*(F_1, F_2)$ is asymmetric-balanced,
\item for every hypergraph $G$ such that $m(G) \leq m_k(F_1, F_2)$ it follows that 
$$G \nrightarrow (F_1, F_2).$$
\end{enumerate}
Then there exists a constant $c > 0$ such that for $p \le c n^{-1 / m_k(F_1, F_2)}$
$$
\lim_{n \rightarrow \infty} \Pr \left[ \Hknp \rightarrow (F_1, \ldots, F_r) \right] = 0.
$$
\end{theorem*}


Note that it is sufficient to prove the statement for the case of two colours. By having more than two colours we can always restrict to only the first two, which avoids a monochromatic copies of $F_i$ for $i \ge 3$. Thus we assume $r =2$ and we call the colours red and blue.

We use a grow-sequences approach very similar to the one in \cite{nenadov2015algorithmic,nenadov2016short}. 
We say an edge is \emph{closed} if it is contained in a copy of
\(F_1\) and a copy of \(F_2\) which are otherwise edge-disjoint, and
\emph{open} otherwise. 
The reason for this distinction is that open edges are easy to colour. Assume $e$ is an open edge and 
there exists a valid \(2\)-colouring (i.e.\ one without red copy of
$F_1$  and  blue copy of $F_2$) for \(H - e\), the hypergraph obtained from \(H\) by removing \(e\)
from the edge set. Then we can extend this colouring to \(H\) by using the fact that \(e\) is open: if there exists a copy of $F_1$ in $H$, say $\hat F_1$, which contains $e$ and is already monochromatic up to $e$ then colour $e$ with blue, and otherwise colour it with red. This clearly avoids a red copy of $F_1$ in $H$. Let us assume, towards a contradiction, that there exists a blue copy of $F_2$. First, note that any such copy has to contain $e$ and therefore $e$ has to be coloured with blue. However, this implies that such copy of $F_2$ does not intersect $\hat F_1$ on any other edge except $e$ (since every other edge of $\hat F_1$ is red), which contradicts the assumption that $e$ is open.

Using the notion of an open edge, we can find a valid \(2\)-colouring of \(\Hknp\) by running the
following algorithm:
\begin{algorithm}[h]
  \DontPrintSemicolon
  \(\hat H := \Hknp\)\;
  \While{\(\hat H\) contains an open edge \(e\)}
  {
    \(\hat H \leftarrow \hat H - e\)\;
  }
  Colour \(\hat H\) without a red $F_1$ and a blue $F_2$\;
  Add the removed edges in reverse order and colour them appropriately.

  \vspace{0.4cm}
  \caption{Colouring algorithm.}\label{algo:col}
\end{algorithm}

Note that the obtained graph $\hat H$ is uniquely defined. Indeed, since if an edge is open in $H$ then it is also open in every $H' \subseteq H$, the order in which we remove open edges is irrelevant. 

Of course, the step ``Colour \(\hat H\)'' in the Algorithm \ref{algo:col} is the difficult one and the main part of the proof is to show that this is indeed possible with high probability. Our strategy is to first show that $\hat H$ can be split into hypergraphs of constant size which can be coloured separately (using properties (i)--(iii)). Then using the bound on $p$ we deduce that every such hypergraph has small density and, finally, from the assumption (iv) we conclude that it can be coloured without red $F_1$ or blue $F_2$. To state this precisely we need a couple of definitions.

\begin{definition}[$(F_1, F_2)$-core]
We say a subgraph $G' \subseteq \hat H$ (where $\hat H$ is the resulting graph obtained by the Algorithm \ref{algo:col}) is an  $(F_1, F_2)$-core if every copy of $F_1$ or $F_2$ in $\hat H$ is either  contained in $G'$ or edge-disjoint from $G'$. 
\end{definition}

Since $\hat H$ is closed by the definition of the algorithm, every $(F_1, F_2)$-core is closed as well. However, to emphasize this fact we shall sometimes explicitly write that the $(F_1, F_2)$-core under consideration is closed.

Next, let $\hat H_1, \ldots \hat H_t$ be a partition of $\hat H$ into edge-disjoint minimal (by subgraph inclusion) closed $(F_1, F_2)$-cores. By the definition of the core and the assumption that $\hat H_i$'s are edge-disjoint, if there exists a valid colouring of each $\hat H_i$ then the same colourings induce a valid colouring of $\hat H_1 \cup \ldots \cup \hat H_t = \hat H$. The following lemma is the key ingredient in the proof of Theorem \ref{thm:asymmetric-0}.

\begin{lemma}
 \label{lemma:boundedcore}
 Let $F_1$ and $F_2$ be as in Theorem \ref{thm:asymmetric-0}. Then there exist constants $c = c(F_1, F_2) > 0$
 and $L = L(F_1, F_2) > 0$ such that if $p \le c n^{-1/m_k(F_1, F_2)}$ then w.h.p. every minimal closed $(F_1, F_2)$-core of $\Hknp$ has size at most $L$.
\end{lemma}


Before we go into the proof of the lemma, we first use it to derive Theorem \ref{thm:asymmetric-0}.

\begin{proof}[Proof of Theorem \ref{thm:asymmetric-0}]
Let $c$ and $L$ be as given by Lemma \ref{lemma:boundedcore}. First, note that w.h.p.\ every hypergraph $G \subset \Hknp$ of size at most $L$ satisfies $m(G) \le m_k(F_1,F_2)$. This can be derived as follows. For every hypergraph $G$ with at most $L$ vertices and $m(G) > m_k(F_1, F_2)$ we have
$$
  1/m_k(F_1, F_2) \ge 1/m(G) + \alpha,
$$
for some $\alpha > 0$. Moreover, there exists $\alpha = \alpha(L)$ such that this holds for all such hypergraphs $G$ simultaneously. Consider some $G$ with $m(G) > m_k(F_1, F_2)$ and let $G' \subseteq G$ be such that $m(G) = e(G')/v(G')$. Note that the expected number of copies of $G'$ in $\Hknp$ is at most 
$$
  n^{v(G')} p^{e(G')} \le n^{v(G')} \left( cn^{-1/m_k(F_1,F_2)} \right)^{e(G')} \le n^{v(G')} n^{-v(G') - \alpha} = n^{-\alpha}.
$$
Therefore, by Markov's inequality we have that $G'$, and therefore $G$, does not appear in $\Hknp$ with probability at least $n^{-\alpha/2}$. Since there are at most $L 2^{L^2}$ such hypergraphs and $L$ does not depend on $n$, by union bound we have that w.h.p none of them appears.

Next, let $\hat H$ be the hypergraph obtained using the Algorithm \ref{algo:col} and let $\hat H_1, \ldots, \hat H_t$ be a partition of $\hat H$ into edge-disjoint minimal closed $(F_1, F_2)$-cores. By Lemma \ref{lemma:boundedcore} we have that each $\hat H_i$ has size at most $L$, and from the previous observation we conclude that $m(\hat H_i) \le m_k(F_1, F_2)$. Now using the property (iv) we obtain a colouring of $\hat H_i$ without a monochromatic $F_1$ and $F_2$ which, by the discussion preceding the proof, gives a valid colouring of $\hat H$. This, in turn, gives a valid colouring of the whole $\Hknp$ which finishes the proof.
\end{proof}

In the rest of this section we prove Lemma \ref{lemma:boundedcore}. The proof of the lemma will rely on a somewhat technical claim which we postpone to the next section.

\begin{proof}[Proof of Lemma \ref{lemma:boundedcore}]
Our strategy is to show that every minimal closed $(F_1, F_2)$-core is either of size at most $L$ or is much larger, namely of size $\Omega(\log n)$. Then, using some further properties of such hypergraphs, we deduce that the latter case does not happen in $\Hknp$. Our main tool in proving this is a procedure (Algorithm \ref{algo:grow-sequence-algo-asymm}) which generates each $(F_1, F_2)$-core in a systematic way.

Let $G$ be some minimal closed $(F_1, F_2)$-core.
 We assume that some arbitrary total ordering on the vertices of $G$ is given. By lexicographic ordering this induces a total ordering on the edges of $G$ as well, i.e. we can always choose a well-defined minimal edge out of any edge set. 
 For the enumeration aspect of the problem we map any minimal closed $(F_1, F_2)$-core $G$ to a sequence of hypergraphs $F^*_i$'s  via Algorithm \ref{algo:grow-sequence-algo-asymm}, where $F^*_i$ is either a member of $\cF^*(F_1, F_2)$ or is isomorphic to $F_1$.
We denote the family $\cF^*(F_1, F_2)$ by $\cF^*$, as it will always be used with respect to $F_1$ 
and $F_2$.

{\LinesNumbered
\begin{algorithm}
  \DontPrintSemicolon
  Let \(F^*_1\) be a copy of \(F_1\) in \(G\)\;
  \(G_1 \leftarrow F^*_1\)\;
  \(i \leftarrow 1\)\;
  \While{\(G_i \neq G\)}{
    \(i \leftarrow i + 1\)\;
    \eIf{\(G_{i-1}\) contains an open edge}{
      \(j \leftarrow\) smallest index \(j < i\) such that \(F^*_j\) contains
      open edges\;
      \(e \leftarrow\) the minimal open edge in \(F^*_j\)\;
      \eIf{there exists a copy \(\hat F_1\) of \(F_1\) in \(G\), not contained in \(G_{i-1}\), which
      contains $e$  \label{line:copy-of-F_1-open} }
      {\(F^*_i \leftarrow \hat F_1\)  \;
      }
      {
      \(F^*_i \leftarrow\) a copy of some hypergraph from \(\mathcal{F}^*\)
      in \(G\) but not in \(G_{i-1}\), which contains \(e\) as an attachment edge\label{line:copy-of-f-star-open}\;
      }
    }{ 
    
    \eIf{there exists a copy \(\hat F_1\) of \(F_1\) in \(G\), not contained in \(G_{i-1}\), which intersects \(G_{i-1}\) in at least one edge  \label{line:copy-of-F_1-closed} }
    { \(F^*_i \leftarrow \hat F_1\)  \; }
    {
         \(F^*_i \leftarrow \)  a copy \(F^*\) of a hypergraph in \(\mathcal{F}^*\) in
              \(G\), not contained in \(G_{i-1}\), such that its attachment edge is contained in $G_{i-1}$
                     \label{line:copy-of-f-star-closed} \;
                     }
    }
    \(G_{i} \leftarrow G_{i-1} \cup F^*_i\)
  }\caption{Decomposing minimal closed $(F_1, F_2)$-cores.}\label{algo:grow-sequence-algo-asymm}
\end{algorithm}

Let us first prove that Algorithm \ref{algo:grow-sequence-algo-asymm} is well defined.
If $G_{i-1}$ has an open edge $e$, then either there is a copy of $F_1$  or $F_2$ not contained in $G_{i-1}$, but which contains $e$. Moreover, note that every copy of $F_2$ is contained in some $F^* \in \cF(F_1, F_2)$, since each edge of $G$ is closed. These observations prove that if $G_{i-1}$ has an open edge, then one of the two requirements of the lines \ref{line:copy-of-F_1-open} and \ref{line:copy-of-f-star-open} are satisfied.

Similarly, if $G_{i-1}$ doesn't have an open edge and $G_{i-1} \neq G$ then there is a copy of $F_1$ or $F_2$ which intersects $G_{i-1}$, but which
is not contained in it. By the previous argument, one of the two requirements of the lines \ref{line:copy-of-F_1-closed} and \ref{line:copy-of-f-star-closed} are satisfied.
 The algorithm is therefore correct and
terminates. 

Note that the sequence $S:= (F^*_1, \ldots, F^*_\ell)$ fully describes a run of the algorithm. We call it a \emph{grow sequence} for $G$ and each $F^*_i$ in it a \emph{step} of the sequence, where $1 \le i \le \ell$. Set $S_i = (F^*_1, \ldots, F^*_i)$ to be the grow sequence consisting of the first $i$ steps. 
 With this we turned the problem of enumerating all minimal closed $(F_1, F_2)$-cores  into the one of enumerating all grow sequences which may appear as output of Algorithm \ref{algo:grow-sequence-algo-asymm}. We aim to estimate the expected number of grow sequences. As already hinted in the beginning of the proof, we show that the algorithm either terminates with a grow sequence of size $O(1)$, or it has to produce a sequence $S$ of size $\Omega(\log n)$. However, in the latter case we show that a subsequence of $S$ truncated after the first $\Theta(\log n)$ steps  does not appear in $\Hknp$, which in turn implies that $S$ does not appear either. We now make this precise.

We distinguish various step types. \(F^*_1\) is the first step. 
Steps chosen in lines \ref{line:copy-of-f-star-open} and \ref{line:copy-of-f-star-closed} are called \emph{regular} if $e(F^*_i) - e(H_i))/(v(F^*_i) - v(H_i)) = m_k(F_1, F_2)$, where $H_i := F^*_i \cap
G_{i-1}$, otherwise they are called \emph{degenerate}. Furthermore, regular step is \emph{open} or \emph{closed}, depending whether it was chosen in line \ref{line:copy-of-f-star-open} or \ref{line:copy-of-f-star-closed}.
 The steps chosen in lines~\ref{line:copy-of-F_1-open} and \ref{line:copy-of-F_1-closed} are by definition
 always degenerate. We say $F^*_1$ is degenerate by definition.

Consider any open regular step \(F^*_i\). By asymmetric-balancedness of $\mathcal{F}^*$, its intersection \(H_i\) with \(G_{i-1}\) is exactly one edge and $F^*_i$ is generic, and thus we can bound the contribution of each such step to the expected number of sequences by
\begin{equation}
k! \cdot n^{v(F^*_i)-k}p^{e(F^*_i)-1}.
\end{equation}
By using $p \le c n^{-1 / m_k(F_1, F_2)}$ and  assumption $(iii)$ from Theorem \ref{thm:asymmetric-0}  we obtain
\begin{align*}
 n^{v(F^*_i)-k} p^{e(F_i^*) - 1} & \le c^{e(F_i^*) - 1} 
n^{v(F^*_i)-k - (e(F_i^*) -1)/m_k(F_1, F_2) }  \le
 c^{e(F_i^*) - 1} .
\end{align*}
From previous it follows
\begin{equation}\label{eq:regular-open-asymm}
   k! \cdot n^{v(F^*_i)-k}p^{e(F^*_i)-1} \leq 
k ! \cdot c^{e(F^*_i)-1} \leq 
  c,
\end{equation}
where the last inequality holds as we may choose constant $c$
small enough, and such that \(\log(c) < -1\) holds.

At every step, we add some new vertices, but never more than the number of vertices in some generic $F^{\flower}$. That means that after $i-1$ steps there are at most $v(F^\flower) \cdot (i-1)$ vertices in $G_{i-1}$.
Hence, if \(F_i^*\) is a  closed regular step, we can bound the contribution to the expected number of sequences by
\begin{equation}\label{eq:regular-closed-asymm}
  k! (v(F^\flower) \cdot (i-1))^{k} n^{v(F^*_i)-k} p^{e(F^*_i) - 1} 
  \leq c  (v(F^\flower) \cdot i)^{k}.
\end{equation}

Now consider the case of a degenerate step \(F^*_i\) which is a copy of some
\(F^*\in \mathcal{F}^*\). By asymmetric-balancedness condition we can
choose a constant \(\alpha_1 > 0\) such that regardless of the choice of \(F^*
\in \mathcal{F}^*\) we have (recall that $H_i = F^*_i \cap
G_{i-1}$)
\begin{equation*}
  v(F^*) - v(H_i) - \frac{e(F^*) - e(H_i)}{m_k(F_1, F_2)} < -\alpha_1.
\end{equation*}
In the case of a degenerate step consisting of just a copy of \(F_1\) (i.e.\
one as in line~\ref{line:copy-of-F_1-open} or \ref{line:copy-of-F_1-closed}) we can choose \(\alpha_2 > 0\) such that
for all \(H_i \subsetneq F_1\), \(e(H_i) \geq 1\), we have
\begin{equation*}
  v(F_1) - v(H_i) - \frac{e(F_1) - e(H_i)}{m_k(F_1, F_2)} < -\alpha_2,
\end{equation*}
by the fact that $F_1$ is strictly balanced with respect to $m_k(\cdot, F_2)$.
Note that this holds even for \(H_i\) being exactly one edge, as \(m_k(F_1 ) >
m_k(F_1, F_2)\). We then set \(\alpha = \min\{\alpha_1, \alpha_2\}\).

With this we obtain an upper bound for the contribution of a degenerate step \(F^*_i\),
\begin{equation}\label{eq:degenerate-asymm}
  \abs{\cF^*} \cdot v(F^*_i)^{v(H_i)}(i\cdot v(F^\flower))^{v(H_i)} \cdot n^{v(F^*_i) - v(H_i)} p^{e(F^*_i) - e(H_i)} \leq
  C i^{v(F^\flower)} n^{-\alpha},
\end{equation}
where $C$ is a suitably chosen constant.

To finish the argument we use the following claim whose proof is presented in the next section.

\begin{claim}\label{lem:open-edges-bound-asymm}
There exist positive constants $C_1, C_2$ (depending on $F_1$ and $F_2$) such that the following holds.
Let \(S\) be a grow sequence of length \(t\).
  \begin{enumerate}[i)]
  \item \label{item:lem-open-edges-bound-1-asymm}If \(S\) contains at most \(d\)
    degenerate steps, then \(t \le d ( 1 + C_2 / C_1) \)
  \item \label{item:lem-open-edges-bound-2-asymm}If a prefix \(S_i\) of \(S\) contains at most \(d\) degenerate steps, then \(S_i\)
    contains no closed regular steps \(F_j\) with \(j > d ( 1 + C_2 / C_1) + 1\).
  \end{enumerate}
\end{claim}

Let $C_1, C_2$ be as in the claim above. 
We can now finish our first moment argument. Set \(\dmax := v(F_1)/\alpha + 1\)
and \(L = (1 + C_2 / C_1)\cdot\dmax +1\). By
Claim~\ref{lem:open-edges-bound-asymm} all sequences longer than \(L\) must
contain at least \(\dmax\) degenerate steps. Set \(\lmax := v(F_1) \log(n) +
d_{\text{max}}+1\).  We consider two cases: the first are those sequences which
have their \(\dmax\)th degenerate step before \(\lmax\). We truncate them
after the \(\dmax\)th step. The second are those sequences whose \(\dmax\)th
step appears after \(\lmax\). We truncate these at length \(\lmax\). Note that
by Claim~\ref{lem:open-edges-bound-asymm} in both of these cases closed regular
steps can only happen in the first \(L\) steps of the sequence.

 
Let us now analyze the first case when  \(\dmax\)th degenerate step happens before \(\lmax\).  Using \eqref{eq:regular-closed-asymm} and \eqref{eq:degenerate-asymm} and summing over all possible steps $t$ when \(\dmax\)th degenerate step happened we bound the expected number of such grow sequences
\begin{multline*}
  \sum_{t = L+1}^{\lmax-1} n^{v(F_1)}
  \binoms{t}{\dmax}\bigl(C t^{v(F^\flower)} n^{-\alpha}\bigr)^{\dmax}
   (v(F^\flower) L)^{kL} 
  =\\ = \mathcal{O}(\polylog(n)\cdot n^{v(F_1)} n^{-\alpha\cdot\dmax}) = 
  \mathcal{O}(\polylog(n)\cdot n^{v(F_1) -v(F_1) -\alpha}) = o(1).
\end{multline*}
Here we bound the contribution of the first step by \( n^{v(F_1)}\), drop the
contribution of \(c < 1\) for all regular steps, and use the fact that only
the first \(L\) steps may be closed regular. 

In the second case, by summing over the number of degenerate steps $d$ before step \(\lmax\), we obtain 
\begin{multline*}
  \sum_{d = 0}^{\dmax}n^{v(F_1)}\binoms{\lmax}{d}
  \bigl(C \lmax^{v(F^\flower)} n^{-\alpha}\bigr)^{d}
   (v(F^\flower) L)^{kL}  c^{\lmax-d-1} 
  \\= \mathcal{O}(\polylog(n)\cdot n^{v(F_1)} c^{\lmax-1-d}) =
  \mathcal{O}(\polylog(n)\cdot n^{v(F_1)(1+\log c)}) = o(1),
\end{multline*}
where the last step holds because $\log c < -1$.
With this we proved that only grow sequences of length at most the constant
\(L\) can appear in \(\Hknp\) for \(p \leq c n^{1/m_k(F_1, F_2)}\).}
\end{proof}


It remains to prove Claim \ref{lem:open-edges-bound-asymm}, which we do in the next section.

\subsection{Bounding grow sequences - Proof of Claim \ref{lem:open-edges-bound-asymm} }

Let us first prove an auxiliary  claim.
\begin{claim}
\label{claim:regular-has-open-edges}
Let $G$ be an arbitrary hypergraph and let $F_1, F_2$ be as in Theorem \ref{thm:asymmetric-0}. Furthermore, let $F$ be an $F_1$-copy which 
intersects $G$ in exactly one edge $e$ in $E(G)$. 
Set $G_F$ to be a hypergraph $(V(G) \cup V(F), E(G) \cup E(F))$.
Then every edge in $E(F) \setminus e$ is open in $G_F$.
\end{claim}
\begin{proof}
Let $e^*$ be an arbitrary edge from $E(F) \setminus e$.
We first prove that any $F_1$-copy that contains $e^*$, must contain all edges from $E(F) \setminus e$. Note that this is equivalent to proving that all $F_1$-copies that contain $e^*$ have exactly one edge intersecting $e(G)$.
In order to arrive at a contradiction let $F'$ be an $F_1$-copy such that $|E(F') \cap E(G)| \ge 2$. Set $F_{\text{new}} = F'[(V(F) \setminus V(G))\cup e]$ and $F_{\text{old}} = F'[V(G)]$. Furthermore, let us denote $F_{\text{new}}^{+e}$ as the hypergraph obtained by adding edge $e$ to $F_{\text{new}}$ (if the edge is already in $F_{\text{new}}$ then 
$F_{\text{new}} = F_{\text{new}}^{+e}$). Note that both  $F_{\text{new}}^{+e}$ and 
$F_{\text{old}}$ are strict subgraphs of $F_1$. Since 
$e(F') = e(F_{\text{new}}^{+e}) + e(F_{\text{old}}) - 1$ and 
$v(F') \ge v(F_{\text{new}}^{+e}) + v(F_{\text{old}}) - k$  we obtain
$$
m_{\ell}(F') = \frac{e(F') - 1}{v(F') - k} \le 
\frac{e(F_{\text{new}}^{+e}) -1 + e(F_{\text{old}}) - 1}
{v(F_{\text{new}}^{+e}) -k + v(F_{\text{old}}) - k}.
$$
Since $F'$ is strictly balanced we know that $m_{k}(F_{\text{new}}^{+e})$ and 
$m_{k}(F_{\text{old}})$ are strictly smaller than $m_{k}(F')$. Together with inequality above, this  implies 
$$
m_{k}(F') \le 
\frac{e(F_{\text{new}}^{+e}) -1 + e(F_{\text{old}}) - 1}
{v(F_{\text{new}}^{+e}) -k + v(F_{\text{old}}) - k} < m_{k}(F') ,
$$
which is a contradiction.

In order to prove the lemma, it is sufficient to prove that any $F_2$-copy that contains $e^*$ must intersect $E(F) \setminus e$ on at least one edge other than $e^*$. Assume $H$ is an $F_2$-copy which contains $e^*$ but otherwise is fully contained in $G$. Let us denote $H_{\text{old}}$ be $H[V(G)]$. 
Since $H\cong F_2$ is strictly $m_k$-balanced, we have
 $m_{k}(H) > \frac{e(H_{\text{old}}) - 1}{v(H_{\text{old}}) - k }$. Therefore we obtain with  $m_{k}(H) \ge 1$ that 
$$
m_k(H) > \frac{e(H_{\text{old}}) - 1 + 1}{v(H_{\text{old}}) - k + 1}.
$$
This is, however, a contradiction as $m_{k}(H) = \frac{e(H) - 1}{v(H) - k }=\frac{e(H_{\text{old}}) }{v(H_{\text{old}})+(v(H)-v(H_{\text{old}})) - k}$ and $H$ contains at least one vertex more than $H_{\text{old}}$. We conclude that any $F_1$-copy and $F_2$-copy that contain $e^*$ intersect on at least two edges and thus $e^*$ is closed.
\end{proof}

Before we continue with the proof, we introduce some definitions.
For any regular step $F^*_i$, we call the edge $e = E(G_{i-1}) \cap E(F^*_i)$ the \emph{attachment edge} of $F^*_i$ and 
the vertices in $V(F^*_i) \setminus V(G_{i-1})$ the \emph{inner vertices} of $F^*_i$.
If $F^*_i$ is a regular step, then $F^*_i$ contains \(
 (e(F_2) - 1) (e(F_1) - 1) \) open edges in $G_i$ (i.e.\ every edge in a petal of $F^*_i$, not contained in the center of $F^*_i$). In fact,
we call a step \(F_j^*\), \(j \leq i\), \emph{fully
  open} in \(S_i\) if it is a regular step   and
no other step \(F_{j'}^*\), \(j < j'\leq i\), intersects any vertex of
\(F_j^*\) which is not in the attachment edge of \(F_j^*\). 
Note that by Claim \ref{claim:regular-has-open-edges} a fully open step $F^*_j$ in $S_i$ contains $(e(F_2) - 1) (e(F_1) - 1)$ open edges in $G_i$.
Finally, we denote by $\reg(S_i)$, $\deg(S_i)$ and $\fo(S_i)$ the number of regular, degenerate
and fully-open steps in $S_i$. The next claim states that a fully open step $F^*_i$ stays fully open if none of the following steps intersects its inner vertices.

\begin{claim}
Let $S = (F^*_1, \ldots, F^*_l)$ be a grow sequence and $F^*_j, j \leq i$ a regular step, for some $i \leq l$. 
Then if no step $F^*_{j'}$, for $j' \in \{ j + 1, \ldots, i\}$, contains an inner vertex of $F^*_j$, then 
$F^*_j$ is fully open in $S_i$.
\label{claim:open-stays-open}
\end{claim}
\begin{proof}
Let $G_{i}$ be the graph obtained at $i$-th step of the grow sequence and let 
$F_2$ be the `center' and $F_1^1, \ldots, F_1^{e(F_2) -1}$ be the `petals' of  $F^*_j$.  Set $$
H_j := G_{i-1} \cup F_2 \cup 
\bigcup_{t \in [e(F_2) - 1] \setminus \{j\}} F^t_1. 
$$
By applying Claim \ref{claim:regular-has-open-edges} to $H_t$ (as $G$) and $F^t_1$ as $(F)$ for all $t \in [e(F_2)]$ we obtain the lemma.
\end{proof}

For $i \ge 1$ let $\kappa(i)$ denote the number of fully open copies "destroyed" by step $i$, i.e. 
$$
\kappa(i) = |\{j < i \colon F^*_j \text{ is fully open in $S_{i-1}$ but not in $S_i$}\}|.
$$ 
The following claims show how  regular and degenerate steps influence $\kappa(\cdot)$.

\begin{claim}
\label{claim:reg-effect}
If $F_i^*$ is a regular step, then $\kappa(i) \leq 1$.
\end{claim}
\begin{proof}
Since the attachment edge of $F^*_i$ can intersect inner vertices of only one regular step, together with Claim \ref{claim:open-stays-open} the proof is concluded.
\end{proof}
\begin{claim}
\label{claim:deg-effect}
If $F_i^*$ is a degenerate step, then $\kappa(i) \leq v(F^*_i)$.
\end{claim}
\begin{proof}
Let us first prove that any vertex from $G_{i-1}$ is an inner vertex of at most one regular copy $F^*_j$ which is fully open in $S_{i-1}$. Let $v \in G_{i-1}$ be an arbitrary vertex from $G_{i-1}$ which is an inner vertex of some regular copy $F^*_j$ such that $F^*_j$ is fully-open in $S_{i-1}$.
By definition of a fully-open copy, no step $F^*_t$, for $j < t < i$, can contain $v$. On the other hand,
 $v \not \in G_{j-1}$ as $v$ is an inner vertex of a regular step.

By Claim~\ref{claim:open-stays-open}, a fully-open step $F^*_j$ in $S_{i-1}$ is a fully-open step in $S_i$ if $F^*_i$ does not intersect inner vertices of $F^*_j$. This together with the observation above proves the claim. 
\end{proof}

\begin{claim}
\label{claim:seq_regular}
Set $d = (e(F_2) - 2).$ 
Let $F^*_i, \ldots, F^*_{i + d}$ be a sequence of consecutive regular steps such that $\kappa(i) = 1$. Then $\kappa(i+1) = \ldots = \kappa(i + d) = 0$.
\end{claim}
\begin{proof}
As $\kappa(i) =1 $ it holds that $F^*_i$ is the first step which intersects the inner vertices of some fully open step $F^*_j, j < i$. From previous observations we have that $F^*_j$ has $(e(F_2) - 1)(e(F_1) - 1)$ open edges before the step $F^*_i$ is made. 
It is sufficient to prove that after steps $F^*_{i}, \ldots, F^*_{i + s}$, for $s < e(F_2) -2$, there is still at least one open edge in $F^*_j$.
Let $F'_1, \ldots, F'_{e(F_2) -1}$ be the petals of $F^*_j$ (which are copies of $F_1$). Note that $F^*_{i+1}, \ldots, F^*_{i + s}$ can
intersect at most $s$ of different petals and without loss of generality let us assume they are 
$F'_1, \ldots, F'_q$, for $q \le s$.
By applying Claim \ref{claim:regular-has-open-edges} to $F'_{q + 1}$ (as $F$) and 
$G_{i+s}$ (as $G$) we obtain that all the edges in $F'_{q+1}$, except for maybe one, are open in $G_{i+s}$. 
Observe that this is the case since we always choose `minimal' open edges (cf.\ Algorithm~\ref{algo:grow-sequence-algo-asymm}) and for two fully open steps in some $S_i$ those edges are smaller that appear first in the grow sequence.
\end{proof}

Using previous three claims we are able to show connection between the number of fully-open steps in $S_i$ and the number of regular and degenerate steps in $S_i$.

\begin{claim}
\label{claim:recurence}
Set $C_1 := 1 - \frac{1}{e(F_2) - 2}$ and $C_2:= v(F_2) + (e(F_2) -1) (v(F_1) - k) + 1$. Then $\fo(S_i) \ge \reg(S_i)\cdot C_1 - \deg(S_i)\cdot C_2$
\end{claim}
\begin{proof}

Set $\phi(i) = \fo(S_i) - \reg(S_i)\cdot C_1 + \deg(S_i)\cdot C_2.$
We prove by induction the following, stronger, claim:  for every $i \ge 1$
 $$
 \phi(i) \ge 
 \begin{cases}
 1 & F^*_i \text{ is a degenerate step,}\\
 0 & F^*_i \text{ is a regular step.}
 \end{cases}
 $$
Since the first step is by definition degenerate the hypothesis is true for $i = 1$.
Assume the claim holds for all $i < i'$.
If $\kappa(i') = 0$, then since $C_1 < 1$ and $C_2 \ge 1$ one can easily check that the claim holds.
Thus, let us assume $\kappa(i') > 0$.
Furthermore, if $F^*_{i'}$ is a degenerate step then by Claim \ref{claim:deg-effect} we know $\phi(i') \ge \phi(i' -1) + 1$. From the induction hypothesis, this implies  $\phi(i') \ge 1$.

From now on we assume $\kappa(i') = 1$ and $F^*_{i'}$ is a regular step. 
Let $j < i'$ be the largest integer such that $\kappa(j) > 0$ or that $F^*_j$ is a degenerate step. 
One easily obtains \begin{equation}
\label{eq:phi}
\phi(i') = \phi(j) + i'- j - 1 - (i' - j) C_1.
\end{equation}
If $F^*_j$ is a degenerate step, then $\phi(j) \ge 1$ and by \eqref{eq:phi} and $C_1 \le 1$ we have $\phi(i') \ge 0$.
Let us consider the case when $F^*_j$ is a regular step. By Claim \ref{claim:seq_regular} we know that $ i' - j \ge e(F_2) - 1$. 
Thus 
\begin{align*}
\phi(i') = \phi(j) + (i- j)(1 - C_1) - 1 > \phi(j) + 1-1 \ge 0,
\end{align*}
where the last inequality holds as $F_2$ has at least three edges.
This concludes the proof.
\end{proof}

We have all the necessary tools to finish the proof of Claim \ref{lem:open-edges-bound-asymm}.
Let $C_1, C_2$ be as in Claim \ref{claim:recurence}.
Let \(S\) be any grow sequence of length \(t\), and \(S_i\), \(i \leq
t\) such that $i = t$ or $F^*_{i+1}$ is a closed regular step. 
Furthermore, let us assume $S_i$ contains at most $d$ degenerate steps.
By construction,  graph $G_i$ does not have an open edge and thus it does not have a fully open copy, i.e. $\fo(S_i) = 0$. 
Using Claim \ref{claim:recurence} we obtain
$$
\fo(S_i)   = 0  \geq \reg(S_i)\cdot C_1 - \deg(S_i)\cdot C_2.
$$
Since there are at most $d$ degenerate steps in $S_i$, by using the previous observation we obtain
$$
d C_2 \ge (i - d) C_1,
$$
which implies $i \le d( 1 + C_2 / C_1)$.
This concludes the proof of Claim \ref{lem:open-edges-bound-asymm}.

\section{Threshold for $\Hknp \rightarrow (K_3^{+k}, \HFii)$}\label{sec:specialcase}

In this section prove Lemma \ref{lemma:0_statement_aux}, one of the ingredients in the proof of Theorem \ref{thm:main_beta}. We do that by applying Theorem \ref{thm:asymmetric-0}. Since it is easy to check that 
property $(i)$ in Theorem \ref{thm:asymmetric-0} holds and  that $K_3^{+k}$ is strictly balanced with respect to $m_k(\cdot, \HFii)$, in the next two sections we verify properties (iii) and (iv).

\subsection {Property $(iv)$}

In this section we prove that property $(iv)$ of Theorem \ref{thm:asymmetric-0} holds for $F_1 = \HFi$ and $F_2 = \HFii$. For convenience we state this as a lemma.

\begin{lemma} \label{lemma:property_b}
For $k \ge 4$ and every $k$-uniform hypergraph $H$ such that $m(H) \le m_k(\HFi, \HFii)$ we have
$$
	H \not \to (\HFi, \HFii).
$$
\end{lemma}

First we introduce some notation and definitions. We denote with $\Tlk$ a $k$-uniform hypergraph on the  vertex set $\{v_1, \ldots, v_\ell\}$ and edges given by $\{v_i, \ldots, v_{i + k - 1}\}$ for each $1 \le i \le \ell - k + 1$. We refer to $\Tlk$ as a \emph{tight path} of size $\ell$. An ordering $(e_1, \ldots, e_{\ell - k +1})$ of the edges of $\Tlk$ is called \emph{natural} if $|e_i \cap e_{i+1}| = k - 1$ for every $1 \le i \le \ell - k$.

The proof of Lemma \ref{lemma:property_b} relies on the following two lemmas. 

\begin{lemma}
Let $k \geq 4$ be an integer and let $H$ be $k$-uniform hypergraph with at most $\lceil \frac{3}{2} (k + 1) \rceil$ edges. Then there exists a $(k-2)$-intersecting  subset $S \subseteq E(H)$ of edges such that $H \setminus S$ does not contain $T^k_{2k}$.
\label{lemma:rdo_lge5}
\end{lemma}

In the proof of Lemma~\ref{lemma:property_b} we require Lemma~\ref{lemma:rdo_lge5} for a value of $k$ that is one smaller than the starting value. Thus for the case $k=4$ we would need Lemma~\ref{lemma:rdo_lge5} for $k=3$. Here we can only show a slightly 
weaker bound on the number of edges that nevertheless requires much work and whose proof can be found in an Appendix (since the lemma deals with hypergraphs on $7$ edges, it can in principle be checked `easily').

\begin{lemma}
\label{lemma:rdo_l4}
Let $H$ be a 3-uniform hypergraph with 7 edges. Then there exists a $1$-intersecting subset $S \subseteq H$  of edges such that $H \setminus S$ does not contain a copy of $T^3_{6}$.
\end{lemma}

Before we prove these two lemmas, we first show how they imply Lemma \ref{lemma:property_b}.

\begin{proof}[Proof of Lemma \ref{lemma:property_b}]
 Let us assume towards a contradiction that Lemma \ref{lemma:property_b} is false for some $k \ge 4$. Then there exists a $k$-uniform hypergraph $H$ with 
 $m(H) \le m_{k}(\HFi, \HFii)$ and $H \rightarrow (\HFi, \HFii)$ such that for every $H'$ obtained from $H$ by removing a single vertex (and all adjacent edges) we have $H' \not \to (\HFi, \HFii)$ (i.e. $H$ is a \emph{vertex-minimal} counterexample). Since we have
$$
 \frac{1}{k |V(H)|} \sum_{x \in V(H)} \deg(x) \le m(H) \le m_{k}( \HFi , \HFii)
$$ 
it follows that there exists a vertex $x \in V(H)$ such that $\deg(x) \le \lfloor k \cdot m_{k}(\HFi, \HFii) \rfloor$.
By the choice of $H$ we know that there exists a colouring of the edges in $H \setminus x$ (i.e. the subgraph induced by the vertex set $V(H) \setminus \{x\}$) without a red $\HFi$ and a blue $\HFii$. We now extend this colouring to the edges of $H$ incident to $x$.

Consider the link hypergraph $H_x$ of the vertex $x$, which is the $(k-1)$-uniform hypergraph with the edges $\left\{e\setminus\{x\}\colon e\in E(H), x\in e\right\}$. 
 Let us assume that there exists a $(k-3)$-intersecting set of edges $S \subseteq E(H_x)$ such that $H_x \setminus S$ does not contain a copy of $T_{2(k-1)}^{k-1}$ (we show later that we can indeed find such a set). Let $R$ and $B$ denote the edges of $H$ obtained by adding back the vertex $x$ to the edges of $S$ and $H_x \setminus S$, respectively.
Note that the set $R$ is $(k-2)$-intersecting since $S$ is $(k-3)$-intersecting. 
We claim that colouring the edges in $R$ with red and the edges in $B$ with blue gives a contradiction to the assumption $H \to (\HFi, \HFii)$:
\begin{itemize}
	\item By the assumption on the colouring of $H \setminus x$ any red copy of $K_3^{+k}$ has to contain $x$ and, therefore, at least two edges from $R$. However, as every two edges of $K_3^{+k}$ intersect on $k-1$ vertices, the existence of such copy would contradict the fact that $R$ is $(k-2)$-intersecting. 

	\item Similarly as in the previous case, a blue copy of $\HFii$ necessarily contains $x$ which implies that the subgraph given by the edges in $B$ contains $T_{2k - 1}^{k}$ (since $B$ is the set of blue edges incident to $x$ and $t_k \ge 2k$). Removing the vertex $x$ from every edge of such a copy gives a copy of $T_{2(k-1)}^{k-1}$ in $H_x \setminus S$, which is a contradiction.
\end{itemize}
To conclude, we obtained a colouring of $H$ which contains no red $K_3^{+k}$ and no blue $\HFii$, thus a contradiction with $H \to (\HFi, \HFii)$.

It remains to prove that we can find an $(k-3)$-intersecting set of edges $S \subseteq H_x$ such that $H_x \setminus S$ does not contain $T_{2(k-1)}^{k-1}$. Recall that
$$
	m_k(K_3^{+k}, \HFii) = \frac{3t_k - 3}{2t_k - k - 1}.
$$
If $k = 4$ then from the choice of $x$ we have
$$
\deg(x) \le \lfloor 4 \cdot m_{4}(K_3^{+4}, C_{8}) \rfloor = 7,
$$
and hence $H_x$ is a $3$-uniform hypergraph with at most $7$ edges. Now we can apply Lemma \ref{lemma:rdo_l4} to $H_x$ to obtain a set $S$ with the desired properties.
Otherwise, if $k \geq 5$ then one can check that
$$
\deg(x) \le \lfloor k \cdot m_{k}(K_3^{+k}, \HFii) \rfloor \le \left \lceil \frac{3}{2} k \right \rceil
$$
and thus $H_x$ is a $(k-1)$-uniform hypergraph with at most $\lceil \frac{3}{2} k \rceil$ edges. Therefore, Lemma \ref{lemma:rdo_lge5} guarantees the existence of the desired set $S$. This concludes the proof.
\end{proof}

In the next two subsections we prove Lemmas~\ref{lemma:rdo_lge5} and~\ref{lemma:rdo_l4}.

\subsubsection{Proof of Lemma \ref{lemma:rdo_lge5}}

We use the following two observations on the structure of tight paths.

\begin{lemma}
\label{lemma:core_l_r}
Let $k\ge 3$ and let $a_0, a_1$ be two different edges of the graph $T_{2k}^{k}$. Let $m := |a_0 \cap a_1|\ge 1$. Then there exist $k - m - 1$ different edges $\{e_1, \ldots, e_{k - m - 1} \} \subseteq E(T_{2k}^{k}) \setminus \{a_0, a_1\}$ such that for each $i \in \{1, \ldots, k - m - 1\}$ the following holds:
\begin{enumerate}
\item $|e_i \cap a_0 \cap a_1| = m$, 
\item $|e_i \cap (a_0 \setminus a_1) | = i$, and 
\item $|e_i \cap (a_1 \setminus a_0) | = k - m - i$.
\end{enumerate}
Moreover, for each edge $e' \in E(T_{2k}^{k}) \setminus \{a_0, a_1, e_1, \ldots, e_{k - m - 1} \}$ there exists $b \in \{0, 1\}$ such that
\begin{enumerate}[$(i)$]
\item $|e' \cap a_0 \cap a_{1}| \le m - 1$, 
\item $|e' \cap (a_b \setminus a_{1-b}) | = 0$, and 
\item $|e' \cap (a_{1-b} \setminus a_{b}) | = k - m$.
\end{enumerate}
\end{lemma}
\begin{proof}
Let $f_1, \ldots, f_{k+1}$ be a natural ordering of the edges of $T_{2k}^k$. As $a_0$ and $a_1$ have exactly $m$ vertices in intersection it follows that there exists an index $i \in \{1, \ldots, m\}$ such that
$a_0 = f_i$. Since the reversed ordering of the edges (i.e.\  $f'_i = f_{k+1-i}$) is also a natural ordering, we can assume that $a_0$ is `to the left' of $a_1$ and thus 
 $a_1 = f_{i  + k - m}$. It is easy to see that the set of edges $e_{k - m - j} = f_{i+j}$ (for $1 \le j \le k  - m - 1$) satisfies properties $1.$--$3$. Consequently, each edge $e' \in E(T_{2k}^k) \setminus \{a_0, a_1, e_1, \ldots, e_{k - m - 1}\} $ has to correspond to an edge from either $\{f_1, \ldots, f_{i-1}\}$ or $\{f_{i + k - m + 1}, \ldots, f_{k + 1}\}$. In both of these cases it is easy to see that properties $(i)-(iii)$ hold.
\end{proof}

\begin{claim}
\label{claim:tl_comp_set}
The largest $(k-2)$-intersecting set in $T^{k}_{k + \ell - 1}$ is of size $\lceil \ell / 2 \rceil$.
\end{claim}
\begin{proof}
Let $\{e_1, \ldots, e_\ell\}$ be a natural order of the edges of $T^{k}_{k + \ell -1}$. Observe that the set
$$
	\{e_1, e_3, \ldots, e_{2 \lceil \frac{\ell}{2} \rceil - 1} \}
$$
is $(k-2)$-intersecting and has size $\lceil \ell / 2 \rceil$. Any set of more than $\lceil \ell / 2 \rceil$ edges contains two edges $e_i$ and $e_{i+1}$, for some $i \in \{1 ,\ldots, \ell -1\}$. Such a set can not be $(k-2$)-intersecting as by definition $|e_i \cap e_{i+1}| = k - 1 $.
\end{proof}

We are now ready to prove Lemma \ref{lemma:rdo_lge5}.

\begin{proof}[Proof of Lemma \ref{lemma:rdo_lge5}]
If $H$ does not contain $T_{2k}^k$ then $S := \emptyset$ satisfies  the required properties. Otherwise, from Claim \ref{claim:tl_comp_set} we have that $H$ contains a $(k-2)$-intersecting set $S$ of size $\lceil (k + 1) / 2 \rceil$. Let $H' := H \setminus S$ and note that $H'$ has at most
\begin{equation}
\left \lceil \frac{3}{2} (k + 1) \right \rceil - \lceil (k + 1) / 2 \rceil = k + 1 \label{eq:H_num_edges}
\end{equation}
edges. If $H'$ is not isomorphic to $T_{2k}^k$ then $S$ satisfies the required properties. Otherwise, let $E(H') = \{h_1, \ldots, h_{k+1}\}$ be a natural order of the edges of $H'$ and label the vertices $V(H') = \{v_1, \ldots, v_{2k} \}$  such that
 $$
 h_i = \{v_i, \ldots, v_{i + k - 1}\}
 $$
for all $1 \le i \le k + 1$. Note that 
$$
	C_1 := \{h_1 , h_3, \ldots, h_{2 \lceil \frac{k + 1}{2} \rceil - 1}\} \quad \text{and} \quad 
	C_2 := \{h_2, h_4, \ldots, h_{2 \lfloor \frac{k+1}{2} \rfloor}\}
$$
are $(k-2)$-intersecting sets of size $|C_1| = \lceil \frac{k + 1}{2} \rceil$ and $|C_2| = \lfloor \frac{k + 1}{2} \rfloor$. We show that $H \setminus C_i$ does not contain $T_{2k}^k$ for some $i \in \{1, 2\}$.

Let us assume towards a contradiction that this is not the case and let $F_1$ and $F_2$ denote arbitrarily chosen copies of $T_{2k}^{k}$ in $C_1 \cup S$ and $C_2 \cup S$, respectively. Since $|C_1 \cup S| \leq k + 2$ there can be at most one edge in $C_1 \cup S$ which is not part of $F_1$.

We first show that $h_1 \in E(F_1)$ implies $h_3 \in E(F_1)$. Let us assume towards the contradiction that $h_1 \in E(F_1)$ and $h_3 \notin E(F_1)$. Since $F_1$ contains all but at most one edge in $C_1 \cup S$, this implies that a subgraph induced by edges 
$ S \cup C_1 \setminus \{h_3\}$ is isomorphic to $T_{2k}^{k}$. Let $E(F_1) = \{f_1, \ldots, f_{k+1} \}$ be a natural order of the edges of $F_1$ and let $j$ be such that $f_j = h_1$. We can assume that $j \leq \lceil k /2 \rceil$ as otherwise we can just reverse the order of the edges. Since $j + 2 \le k$ for $k \ge 5$, the edges $f_{j+1}$ and $f_{j+2}$ are such that $$
|f_{j+1} \cap f_{j+2}| = k - 1 , \quad
|f_{j+1} \cap f_j| = k -1, \quad and \quad 
|f_{j+2} \cap f_j| = k -2.
$$
However, no edge from $C_1 \setminus \{h_1, h_3\}$ intersects $h_1$ on more then $k - 4$ vertices, which implies that both $f_{j+1}$ and $f_{j+2}$ are contained in $S$. As $S$ is a $(k-2)$-intersecting set and 
$|f_{j+1} \cap f_{j+2}| = k - 1$, this gives a contradiction. To conclude, we showed that if $h_1 \in E(F_1)$ then $h_3 \in E(F_1)$.

Note that the previous observation together with the fact that $F_1$ contains all but at most one edge in $C_1 \cup S$ implies that either $\{h_1, h_3\} \subseteq E(F_1)$ or $\{h_3, h_5\} \subseteq E(F_1)$ (or both).
We only consider the first case as the latter follows by a symmetric argument.
From Lemma \ref{lemma:core_l_r} with $a_0 = h_1$ and $a_1 = h_3$ (and $m = k - 2$) we conclude that there exists an edge $x \in E(F_1)$ such that 
\begin{equation}
\label{eq:xobs}
|x \cap h_1 \cap h_3| = k - 2, \quad
|x \cap (h_1 \setminus h_3)| = 1 \quad \text{and} \quad
|x \cap (h_3 \setminus h_1)| = 1.
\end{equation}
Therefore we have $|x \cap h_1| = k - 1$ and as $C_1$ is a $(k-2)$-intersecting set we have $x \in S$.

Next, let us look at $F_2$. As $|C_2| = \lfloor (k+1) / 2 \rfloor$ we have $|S \cup C_2| \le k + 1$. Therefore, $E(F_2) = S \cup C_2$ and from $x \in S$ we conclude $x \in E(F_2)$. Using \eqref{eq:xobs} we obtain
\begin{align*}
& |x \cap h_2 \cap h_4| = k - 3 + I(v_{k+1}), \\
& |x \cap (h_2 \setminus h_4)| = 1 + I(v_2), \quad \text{and} \quad \\
& |x \cap (h_4 \setminus h_2)| =  I(v_{k + 2}) + I(v_{k + 3}),
\end{align*}
where 
$$
I(v) := 
\begin{cases}
1 & v \in x, \\ 
0 & \text{otherwise.}
\end{cases}
$$
Observe that by \eqref{eq:xobs} we also have $V(x) \subseteq V(h_1) \cup V(h_3)$ and since $v_{k + 3} \notin V(h_1) \cup V(h_3)$ we get $I(v_{k + 3}) = 0$.
Moreover, as $|x \cap (h_3 \setminus h_1)| = 1$ and $h_3 \setminus h_1 = \{v_{k+1}, v_{k+2}\}$ we conclude 
$$
	I(v_{k+1}) = 1 - I(v_{k + 2}).
$$
These observations imply that there are only two possibilities:
\begin{itemize}
\item $|x \cap h_2 \cap h_4| = k - 2$ and $|x \cap (h_4 \setminus h_2)| = 0$, or
\item $|x \cap h_2 \cap h_4| = k - 3$, $
|x \cap (h_4 \setminus h_2)| = 1 $ and $|x \cap (h_2 \setminus h_4)|  \geq 1.$
\end{itemize}
By applying Lemma \ref{lemma:core_l_r} with $a_0 = h_2$ and $a_1 = h_4$ (and $m = k - 2$) one can see that 
no edge from $F_2 \setminus \{h_2, h_4\}$ can satisfy either of the two possibilities. This gives a contradiction with the assumption that both $H \setminus C_1$ and $H \setminus C_2$ contain $T_{2k}^k$, which concludes the proof of the lemma.
\end{proof}

\subsection {Property $(iii)$}
In this subsection we prove that property $(iii)$ (the asymmetric-balancedness) from Theorem \ref{thm:asymmetric-0} holds for $F_1 = \HFi$ and $F_2 = \HFii$. For convenience of the reader, we state it as a lemma.

\begin{lemma}\label{lemma:prop_c}
 The family $\mathcal{C}^*=\cF^*(\HFi,\HFii)$ is
  asymmetric-balanced for every $k \ge 4$ (see Definition \ref{def:asymmetric:balanced}).
\end{lemma}

Consider some graph $C^*\in\cC^*$. Recall that by the definition of $\mathcal{F}^*(K_3^{+k}, \HFii)$ there exists a subgraph $\HFii^* \subseteq C^*$ isomorphic to $\HFii$ and an ordering $e_0, \ldots, e_{t_k -1}$ of the edges of $\HFii^*$ such that for each $1 \le i \le t_k - 1$ there exists a subgraph $F^i \subseteq C^*$ which contains an edge $e_i$ and is isomorphic to $\HFi$. Recall that we refer to the edge $e_0$ as an \emph{attachment} edge. For simplicity, we assume that the vertices of $\HFii^*$ are labeled with numbers $\{0, \ldots, t_k - 1\}$ such that $e_i=\{i,i+1,\ldots,k-1+i\}$ (where all additions are modulo $t_k$) for $0 \le i \le t_k - 1$.  Moreover, let $v_i$ denote the vertex such that $e_i\cup\{v_i\}=V(F^i)$ (for $1 \le i \le t_k - 1$). Notice that $v_i$s need not be distinct.

We make a few observations that will lead us to a crucial calculation.

\begin{claim}\label{claim:outside_vertices}
Let $I\subseteq \{i\colon v_i\not\in V(\HFii^*)\}$ and $S=\{v_i\colon i\in I\}$. Then the number of edges from $C^*$ incident to $S$ is at least $|I|+|S|$.
\end{claim}
\begin{proof}
 Set $s:=|S|$ and  let $u_1$, \ldots, $u_s$ be the  vertices of $S$. Since every edge in $C^*$ contains at most one vertex in $S$, we are interested in estimating  $\sum_{j=1}^s \deg_{C^*} (u_j)$.  For every  vertex $u_j$ we denote by $W(j)$ the set of the indices $i\in I$ such that  $F^i$ contains $u_j$ and we set $w(j):=|W(j)|$. Observe that the sets $W(j)$ partition $I$.

Consider two distinct indices $i_1$ and $i_2$ from $W(j)$, for some $1 \le j \le s$. First, note that $F^{i_1}$ and $F^{i_2}$ have at most one edge in intersection. Otherwise we would have $V(F^{i_1}) = V(F^{i_2})$ which implies $e_{i_1} = e_{i_2}$, contradicting $i_1 \neq i_2$. Moreover, if $F^{i_1}$ and $F^{i_2}$ share an edge then, because any two edges of $K^{+k}_3$ have intersection of size $k-1$, we necessarily have that $e_{i_1}$ and $e_{i_2}$ are consecutive edges, i.e. $|i_1 - i_2| = 1$ (note that $0 \notin W(j)$). 

Let $W(j) = \{j_1, \ldots, j_q\}$ where $j_i \le j_{i+1}$. Using the previous observation we estimate $\deg_{C^*}(u_j)\ge w(j)+1$ by counting the contribution of each $F^{j_i}$ in the increasing order. In particular, $F^{j_1}$ contributes two edges and every further $F^{j_i}$ at least one new edge. This is clearly true for $i < q$. If $F^{j_q}$ does not contribute any new edge then from the previous observation we conclude that $e_{j_{q - 1}}, e_{j_q}$ and $e_{j_1}$ are consecutive edges. This, however, cannot be because $0 \notin W(j)$. To conclude, we obtain
$$
	\sum_{j=1}^s \deg_{C^*} (u_j)\ge s+\sum_{j=1}^s w(j)=|S|+|I|.
$$
\end{proof}

\begin{claim}\label{claim:three_copies}
Suppose an edge $e \in C^*$ belongs to $F^i, F^j$ and $F^\ell$, for some $1 \le i < j < \ell \le t_k - 1$. Then the edges $e_i$, $e_j$ and $e_\ell$ 
are consecutive (i.e. $j = i + 1$ and $\ell = i + 2$) and
$$
	e = \{i, i + 2, \ldots, i + k\} \; \text{ or } \; e = \{i + 1, \ldots, i + k - 1, i + k + 1\} \; \text{ or } \; e = e_{i+1}.
$$ 
In particular, we have $e \subseteq V(\HFii^*)$.
\end{claim}
\begin{proof}
  Since any two edges of $K^{+k}_3$ have intersection of size $k-1$ it follows that $e_i$, $e_j$ and $e_\ell$ must pairwise intersect in at least $k-2$ vertices. This is only possible if these edges are three consecutive edges on the cycle $\HFii^*$.

  If $e \in E(\HFii^*)$ then from $|e \cap e_{i + j}| \ge k - 1$ for $j \in \{0, 1, 2\}$ we conclude $e = e_{i+1}$. Next, suppose that $e \not\in E(\HFii^*)$. Note that then $e$ has to contain $e_i \cap e_{i+2}$. This can be seen as follows: from $|e \cap e_i| = |e \cap e_{i+2}| = k - 1$ we have $|e \cap (e_i \cap e_{i+2})| \ge k - 3$ (i.e. $e$ can `miss' at most one vertex from $e_i \cap e_{i+2}$). However, if $|e \cap (e_i \cap e_{i+2})| = k - 3$ then necessarily $e_i \setminus e_{i+2} \subset e$ and $e_{i+2} \setminus e_i \subset e$ which implies $|e| \ge k + 1$, thus a contradiction.  Consequently, exactly one vertex from $\{i, i + 1\}$ and one vertex from $\{i + k, i + k + 1\}$ belongs to $e$. Moreover, from $|e \cap e_{i+1}| = k - 1$ we deduce $\{i, i + k + 1\} \not \subset e$ and from $e \not \in E(\HFii^*)$ we have $\{i + 1, i + k\} \not\subset e$ (as otherwise $e = e_{i+1}$). This leads to the two remaining shapes of $e$.
\end{proof}

Given $j\in[t_k-1]$ such that $V(F^j) \subseteq V(\HFii^*)$, we denote by $E^j$ the set of the non-cycle edges from $F^j$, i.e. 
$$
	E^j:= \{e \in E(F^j) \; : \; e \not\in E(\HFii^*) \; \text{ and } \; e\subseteq V(\HFii^*)\}.
$$
\begin{claim}\label{claim:structure_three_copies}
If an edge $e\not\in \HFii^*$ belongs to $F^i, F^{i+1}$ and $F^{i+2}$ then there is an edge $e'$ in $E^i \cup E^{i+1} \cup E^{i+2}$ such that $\{i, i + k + 1\} \subseteq e'$ and $e'$ belongs to at most two  hypergraphs $F^j$. 
\end{claim}
\begin{proof}
From Claim \ref{claim:three_copies} we know that either $e = \{i, i + 2, \ldots, i + k\}$, $e = \{i + 1, \ldots, i + k-1, i + k + 1\}$ or $e = e_{i+1}$. Assume first that 
$e =  \{i + 1, \ldots, i + k-1, i + k + 1\}$.  Then the edge $e' \in  E^{i} \setminus \{e\}$ must contain $i$ and $i+k+1$. 
If $e = \{i, i + 2, \ldots, i + k\}$ then the edge $e' \in  E^{i+2} \setminus \{e\}$ must contain $i$ and $i+k+1$. 
 Again from Claim \ref{claim:three_copies} we conclude that if $e'$ is contained in three different $F^j$s then it cannot contain both $i$ and $i + k + 1$, thus a contradiction. Therefore, $e'$ is the desired edge.
\end{proof}

\begin{claim}\label{claim:cycle_k3_copy}
Let $F^i$ be such that $V(F^i)\subseteq V(\HFii^*)$. If $|E^i|=1$ 
then either $e_{i-1}$ or $e_{i+1}\in E(F^i)$.
\end{claim}
\begin{proof}
This follows immediately from the fact that  any two edges of $K^{+k}_3$ have intersection of size $k-1$.
\end{proof}

For a given subset $X\subseteq \{k,\ldots,t_k-1\}$ we define $E_X$ to be the set of all edges $e$ of $\HFii^*$ such that either the leftmost vertex of $e$ lies in $X$ or $\min X\in e$, i.e.
$$
	E_X = \{e_i\colon i \in X\} \cup \{e \in E(\HFii^*) \colon \min X \in e \}.
$$

\begin{claim}
\label{claim:cycle_vertices}
Let $X \subseteq \{k,\ldots,t_k-1\}$, and let 
$$
	I = \{ i \colon e_i\in E_X \; \text{ and } \; v_i\not\in V(\HFii^*) \}
$$
and $I'\subseteq I$. If $|X|>3$ then the number of edges in $C^*$ that intersect $X$ and   are either contained in $V(\HFii^*)$ or intersect $\{v_i\colon i \in I'\}$ is at least 
\begin{equation} \label{eq:crazy_bound}
	(|X|+k-1)+\frac{|X|+k-1-|I|}{2}+\frac{|I'|}{2}.
\end{equation}
\end{claim}
\begin{proof}
We write $X=\{x_1<\ldots<x_m\}$ ($m\ge 4$). There are exactly $m$ edges in $\HFii^*$ whose leftmost vertex lies in $X$ (namely $e_{x_1}, \ldots, e_{x_m}$). Moreover, there are further $k-1$ edges $e \in E(\HFii^*)$ with $x_1 \in e$ (note that from $x_1 \ge k$ we have that these edges are different from $e_{x_i}$) and therefore $|E_X| = |X| + k - 1$. Clearly, by the definition we have that each $e \in E_X$ intersects $X$ and $e \subseteq V(\HFii^*)$ thus this establishes the first group of summands in \eqref{eq:crazy_bound}.

Next, we estimate those edges which contain a vertex from $\{v_i\colon i \in I'\}$ and $X$. 
Since $v_i \not \in V(\HFii^*)$ and $e_i \cap X \neq\emptyset$ for $i \in I'$ (this follows from $e_i \in E_X$), we deduce that at least one of the edges from $E(F^{i})\setminus\{e_i\}$ must intersect $X$ and contain $v_i$ (recall that in $K_3^{+k}$ any pair of vertices is connected by an edge).
 Moreover, from Claim \ref{claim:three_copies} we have that every such edge $e$ belongs to at most two $F^{i}$s, as otherwise we have $v_i \in e \subseteq V(\HFii^*)$ which contradicts $v_i \not \in V(\HFii^*)$. This shows that the number of edges incident both to $X$ and $\{v_i\}_{i\in I'}$ is at least $|I'|/2$, which establishes the third summand in \eqref{eq:crazy_bound}.

It remains to show that there are at least
\begin{equation} \label{eq:second_sum}
	\frac{|X| + k - 1 - |I|}{2}
\end{equation}
edges $e \in E(C^*) \setminus E_X$ such that $e \cap X \neq \emptyset$ and $e \subseteq V(\HFii^*)$. Let
\begin{align*}
	E &:= \{e_i \in E_X \; : \; v_i \in V(\HFii^*) \}, \\
	E_1 &:= \{ e_i \in E_X \; : \; v_i \in V(\HFii^*) \; \text{ and } \; \exists e \in E^i \text{ such that } e \cap X = \emptyset \}, \\
	E_2 &:= \{ e_i \in E_X \; : \; v_i \in V(\HFii^*) \; \text{ and } \;  \forall e \in E^i \text{ we have } e \cap X \neq \emptyset \}.
\end{align*}
Observe that $E=E_1\dot\cup E_2$ holds.
Furthermore set
\begin{align*}
	E' &:= \bigcup_{e_j \in E_2} E^j,\\
	\hat E &= \{ e \in E(\HFii^*) \setminus E_X \; : \; e \cap X \neq \emptyset \}.
\end{align*}
Note that $E' \cap \hat E = \emptyset$ (recall that $E^j \cap E(\HFii^*) = \emptyset$ by the definition). Moreover, for every edge $e \in E' \cup \hat E$ we have $e \cap X \neq \emptyset$ and $e \not \in E_X$.  Here the first property follows directly from the definition of the sets $E'$ and $\hat{E}$, whereas the second one follows from the definition of the sets $E^j$ and $\hat{E}\cap E_X=\emptyset$.

Moreover, for every $e_i \in E$ we have $i \not \in I$, thus 
\begin{equation} \label{eq:E_1E_2}
	|E_1| + |E_2| = |E| = |E_X| - |I| = |X| + k - 1 - |I|.
\end{equation}
 We estimate the sizes of $E'$ and $\hat E$ separately. 

\begin{itemize}
	\item $|E'| \ge |E_2| / 2$

Let $S_3$ be the set of those edges from $E'$ that belong to at least three sets $E^j$  for some $j$ such that $e_j \in E_2$. From Claims~\ref{claim:three_copies} and~\ref{claim:structure_three_copies} we deduce that if an edge $e \in E'$ belongs to some $F^{i_1}, F^{i_2}$ and $F^{i_3}$ then there is another edge $e' \in E'$ which also belongs to $E^{i_1}\cup E^{i_2}\cup E^{i_3}$  but lies in at most two copies of some $F^{i}$. Therefore, the number of $E^j$s (where $j\in \{i\colon e_i\in E_2\}$) such that $E^j \subseteq  S_3$ is at most $2|S_3|$
(as otherwise there is an edge $e \in E^{i_1} \cup E^{i_2} \cup E^{i + 3} \subseteq S_3$ which is contained in only two different $F^i$'s, which is a contradiction).
  It follows that $|E'| \ge \frac{|E_2|-2|S_3|}{2}+|S_3|=|E_2|/2$ holds.

	\item $|\hat{E}| \ge |E_1|/2$

By the definition of $E_1$, \emph{any} edge  $e_i\in E_1$ intersects $X$ in \emph{exactly} one vertex. In what follows we will construct a function $f\colon E_1\to \hat{E}$ where every edge of $\hat{E}$ has at most two preimages. This will then imply  $|\hat{E}| \ge |E_1|/2$.

 Given an edge $e\in E_1$ let $x_i$ be the only vertex from $X$ that is contained in $e$. If $e\neq e_{x_1-k+1}$  then $x_i$ is not the rightmost vertex in the edge $e$. If in addition $i<m$ then we define another edge $f(e)$ as follows: let $f(e)\in E(\HFii^*)$  be an edge which contains $x_{i+1}$, where $x_{i+1}$ is preceded by as many vertices as there are vertices in $e$ that come after $x_i$. Clearly, $f(e)\not\in E_X$ and $f(e)\cap X\neq\emptyset$.  
We observe that such edges $f(e)$ are all distinct and in particular $f\colon E_1\setminus\{e_{x_1-k+1},e_{x_m}\}\to \hat{E}$ is injective.  Thus, $|\hat{E}|\ge |E_1\setminus\{e_{x_1-k+1},e_{x_m}\}|$ and if $|E_1\setminus\{e_{x_1-k+1},e_{x_m}\}|\ge |E_1|/2$ then we are done.

 Therefore, it remains to consider the cases when $e_{x_1-k+1}\in E_1$ or $e_{x_m}\in E_1$ holds, where we will define some more edges of the `type' $f(e)$. This amounts to a somewhat tedious case distinction.  If $e_{x_1-k+1}\in E_1$ then we first additionally assume that  either $(x_1+1)\not\in X$ or $(x_1+2)\not\in X$ is the case.  In particular, if $(x_1+1)\not\in X$ then we define $f(e_{x_1-k+1}):=e_{x_2-1}$ and if $(x_1+2)\not\in X$ (and $(x_1+1)\in X$ is) then we set  
 $f(e_{x_1-k+1}):=e_{x_3-1}$.  Similarly, if $e_{x_m}\in E_1$ and $(x_m-1)\not\in X$ then we set $f(e_{x_m}):=e_{x_m-1}$ and if $(x_m-2)\not\in X$ (and $(x_m-1)\in X$ is) we set $f(e_{x_m}):=e_{x_{m-1}-1}$. Notice that $f(e_{x_1-k+1})\neq f(e_{x_m})$ unless $m=4$, $x_2=x_1+1$, $x_4=x_3+1$ and $(x_1+2)\not \in X$ (recall that by assumption $m \ge 4$).
  In any case, it  is easy to check that at most two $f(e)$s are pairwise equal. Therefore, $|\hat{E}|\ge |E_1|/2$.
 
 Finally we treat the case where at least one of the remaining options holds: $e_{x_1-k+1} \in E_1$, $(x_1 + 1), (x_1 + 2)\in X$  or $e_{x_m}\in E_1$ and $(x_m-1), (x_m-2)\in X$. 
 If $e_{x_1-k+1}\in E_1$ and $(x_1 + 1), (x_1 + 2)\in X$, it follows from Claim~\ref{claim:structure_three_copies} and $e_{x_1-k}\not \in E_X$, that none of the edges from $E^{x_1-k+1}$ 
lies in three copies $F^i$ (where all $i\in \{j\colon e_j\in E\}$) -- otherwise these have to be $F^{x_1-k+1}$, 
$F^{x_1-k+2}$ and $F^{x_1-k+3}$. But then with Claim~\ref{claim:three_copies} it follows that $F^{x_1-k+1}$ intersects $X$ in two vertices, which implies $e_{x_1-k+1}\in E_2$ (a contradiction). Moreover, it holds by definition that $|E^{x_1-k+1}|\in\{1,2\}$. And it follows further from Claim~\ref{claim:cycle_k3_copy} that if $|E^{x_1-k+1}|=1$ then either $e_{x_1-k}$ or $e_{x_1-k+2}$ lies in $F^{x_1-k+1}$, which again implies that $e_{x_1-k+1}\in E_2$. Thus, we have  $|E^{x_1-k+1}|=2$. Next we delete one edge from $E^{x_1-k+1}$ which  doesn't intersect $X$ (and still denote the set by $E^{x_1-k+1}$). Similarly, if $e_{x_m}\in E_1$ and $(x_m-1), (x_m-2)\in X$ then none of the edges from $E^{x_m}$ lies in three copies $F^i$ (where all $i\in \{j\colon e_j\in E\}$). Again we have $|E^{x_m}|=2$ and we  
delete from $E^{x_m}$ the edge that doesn't intersect $X$. In each of the cases, we add $e_{x_1-k+1}$ (resp.\ $e_{x_m}$) to $E_2$ (and remove them from $E_1$) and the edges from $E^{x_1-k+1}$ (resp.\ $E^{x_m}$) to $E'$ (but keeping the same notation). Now, a short meditation reveals that the same argumentation as above applies to these slightly altered sets $E_1$ and $E_2$ to show  $|E'| \ge |E_2| / 2$ and $|\hat{E}| \ge |E_1|/2$. Indeed, the inequality $|E'| \ge |E_2| / 2$ holds since  edges $e_{x_1-k+1}$ and $e_{x_m}$ belong to at most two copies of $F^j$ (with $j\in \{i\colon e_i\in E_2\}$) and thus the estimate $|E'| \ge \frac{|E_2|-2|S_3|}{2}+|S_3|=|E_2|/2$ remains valid. Whereas the inequality $|\hat{E}| \ge |E_1|/2$ holds since we exclude $e_{x_1-k+1}$ and/or $e_{x_m}$ from $E_1$ that leads to a simpler function $f\colon E_1\to \hat{E}$, whose property that each element has at most two preimages remains valid.

\end{itemize}




The claim now follows from \eqref{eq:E_1E_2} and previously obtained bounds.
\end{proof}

Now we are ready to prove Lemma \ref{lemma:prop_c}.

\begin{proof}[Proof of Lemma \ref{lemma:prop_c}]
Consider some $C^* \in \mathcal{C}^*$ and suppose $H \subseteq C^*$ contains the attachment edge $e_0$ and $v(H) < v(C^*)$. Let
$$
	S_C := V(C^*) \setminus (V(\HFii^*) \cup V(H)) \quad \text{and}
	\quad S_H := V(H) \setminus V(\HFii^*)
$$
be the set of `outside' vertices of $C^*$ partitioned into those which are contained in $H$ and the rest. 
 Furthermore, set
$$
	I:=\{i \colon  v_i \not \in V(\HFii^*) \cup V(H)\}.
$$
Clearly, $S_C = \{v_i\colon i \in I\}$. Furthermore, let $X\subset V(\HFii^*)$ be such that $S_C \dot\cup X = V(C^*) \setminus V(H)$, i.e.
$$
	X:= V(\HFii^*) \setminus V(H).
$$
Recall that $E_X$ consists of  those edges $e$ of $\HFii^*$ such that either their leftmost vertex of $e$ lies in $X$ or $\min X\in e$. 
We aim to lower bound $e(C^*) - e(H)$  in terms of $|S_C|$ and $|X|$. Since $v(C^*) - v(H) = |S_C| + |X|$ this enables us to bound the ratio  $(e(C^*) - e(H))/(v(C^*) - v(H))$ from below.

Assume first that $|X|\ge 4$. Let us denote by $E_1 \subseteq E(C^*)$ the subset of all edges which intersect $S_C$. From Claim~\ref{claim:outside_vertices} we obtain 
\begin{equation}
	\label{eq:bound_E_1}
	|E_1| \ge |I|+|S_C|.
\end{equation}
Note that from the definition of $S_C$ we have that no edge $e \in E(H)$ intersects $S_C$, and thus it holds $E_1 \subseteq  E(C^*) \setminus E(H)$.
Let us denote 
$$
	J:= \{j \colon v_j\not\in V(\HFii^*) \text{ and }  e_j\in E_X\}.
$$
From Claim \ref{claim:cycle_vertices} we know that the number of edges $e\in E(C^*) 
\setminus E(H)$ that intersect $X$ and are either 
contained in $V(\HFii^*)$ or intersect $\{v_i\colon i\in J \setminus I\} \subseteq S_H$ is at least 
\begin{equation} \label{eq:bound_E_2}
	|X|+k-1+\frac{|X|+k-1-|J|}{2}+\frac{|J\setminus I|}{2}.
\end{equation} 
Let us denote such set of edges with $E_2$.
Note that in order to apply Claim \ref{claim:cycle_vertices} we also need that $X \cap \{0, \ldots, k - 1\} = \emptyset$, which follows from the assumption $e_0 \in E(H)$. 
We have $E_2 \cap E(H) = \emptyset$ as each edge in $E_2$ intersects $X$ which is a set disjoint from $V(H)$. Furthermore, we claim that $E_2$ is disjoint from $E_1$. 
Every edge in $C^*$ intersects the set of `outside' vertices $V(C^*) \setminus V(\HFii^*)$ on at most one vertex. Let $e' \in E_2$  be an arbitrary edge from $E_2$. If $e'$ is 
contained in $V(\HFii^*)$ then $e' \not \in E_1$. Otherwise, we know that $e'$ 
intersects $\{v_i\colon i\in J \setminus I\} \subseteq S_H$. As $S_H$ is disjoint from $S_C
$ we  conclude $e' \notin E_1$ and thus $E_1 \cap E_2 = \emptyset$.
Having this observation, together with the fact 
$E_1 \cup E_2 \subseteq E(C^*) \setminus E(H)$, we bound $e(C^*) - e(H)$  by 
using~\eqref{eq:bound_E_1} and~\eqref{eq:bound_E_2} as follows, 
\begin{equation}\label{eq:lower_bound_edges}
e(C^*) - e(H) \ge |E_1| + |E_2| \ge |I|+|S_C|+|X|+k-1+\frac{|X|+k-1-|J|}{2}+\frac{|J\setminus I|}{2}.
\end{equation}
On the other hand we have $v(C^*)-v(H)=|S_C|+|X|$. Therefore we obtain
\begin{align*}
	\frac{e(C^*) - e(H)}{v(C^*) - v(H)} &\ge \frac{|I|+|S_C|+|X|+k-1+\frac{|X|+k-1-|J|}{2}+\frac{|J\setminus I|}{2}}{|S_C|+|X|} \\
		&= 1+ \frac{2|I| + 3 (k-1) + |X|-|J| + |J\setminus I|}{2(|S_C|+|X|)} \\	
		&\ge 1 + \frac{|I| + 3(k-1) + |X| - |J| + |J|}{2(|S_C| + |X|)} \\
		&\stackrel{(|I| \ge |S_C|)}{\ge} 1 + \frac{|S_C| + |X| + 3(k-1)}{2(|S_C| + |X|)} \\
		&= \frac{3}{2} + \frac{3(k-1)}{2(|S_C|+|X|)}.
\end{align*} 
By comparing this with
$$
	m_k(K^{+k}_3, \HFii)=\frac{3t_k-3}{2t_k-k-1}=\frac{3}{2}+\frac{3(k-1)}{2(2t_k-k-1)},
$$
we see that 
$$
	\frac{e(C^*) - e(H)}{v(C^*) - v(H)}\ge \frac{3}{2}+\frac{3(k-1)}{2(|S_C|+|X|)}\ge \frac{3}{2}+\frac{3(k-1)}{2(2t_k-k-1)} = m_k(\HFi, \HFii).
$$
The second inequality follows from $|X| = |V(\HFii^*) \setminus V(H)| \le t_k - k$ (since $H$ contains $e_0$) and $|S_C| = |V(C^*) \setminus (V(\HFii^*) \cup V(H))| \le |V(C^*) \setminus V(\HFii^*)| \le t_k - 1$. Therefore, the equality is possible only if $|X|=t_k - k$ and $|S_C|=t_k - 1$, in which case $C^*$ is generic and $H$ consists only of the attachment edge. 

It remains to consider the case where $|X|\in \{0, 1,2,3\}$. Suppose $X = \emptyset$. Then $v(C^*) - v(H) = |S_C| \ge 1$. From Claim \ref{claim:outside_vertices} we have that there are at least $|I| + |S_C| \ge 2|S_C|$ edges $e \in C^*$ which intersect $S_C$, and by the definition none of these edges belongs to $H$. Therefore $e(C^*) - e(H) \ge 2|S_C|$ and we get
$$
	\frac{e(C^*) - e(H)}{v(C^*) - v(H)}\ge 2\overset{t_k\ge 2k}{>}m_k(K^{+k}_3, \HFii).
$$

Finally, suppose $1 \le |X| \le 3$. There are at least $|X|+k-1$ edges in $E(\HFii^*) \setminus E(H)$ which intersect $X$. In particular, those are $e_{x}$ for each $x \in X$ and $k-1$ edges preceding $e_{x_1}$ where $x_1 = \min X$ (since $x_1 \ge k$ there is no double counting). Moreover, by the same argument as in the previous case we have at least $|I| + |S_C| \ge 2|S_C|$ edges which intersect $S_C$. This gives
$$
\frac{e(C^*) - e(H)}{v(C^*) - v(H)}\ge \frac{|X|+k-1+2|S_C|}{|X|+|S_C|}\ge 2+\frac{k-1-|X|}{|X|+|S_C|}\overset{k\ge 4}{\ge} 2\overset{t_k \ge 2k}{>}m_k(K^{+k}_3, \HFii),
$$
as required.
\end{proof}

\section{Concluding remarks}\label{sec:conclusion}
The $1$-statement in Theorem~\ref{thm:asymmetric-1} requires $F_1$ be strictly balanced with respect to $m_k(\cdot, F_2)$. We use this condition in Lemmas~\ref{lemma:funionf} and~\ref{lemma:janson} to bound the expected number of pairs of distinct copies of $F_1$ that share at least one edge and to apply Janson's inequality. This is exploited subsequently in the proof of Theorem~\ref{thm:asymmetric-1} in the probability estimate~\eqref{eq:H_not_Ramsey1} via~\eqref{eq:E_i}. At the expense of an additional log-factor we can drop the condition on the strict balancedness and prove the following. 
\begin{theorem} \label{thm:asymmetric-general}
Let $r \ge 2$ and $F_1, \ldots, F_r$ be $k$-uniform hypergraphs such that $m_k(F_1) \geq m_k(F_2) \geq \dotsb \geq m_k(F_r)>0$. Then there exists a constant $C > 0$ such that for $p \ge Cn^{-1/m_k(F_1, F_2)}\log n$ we have
$$
  \lim_{n \rightarrow \infty} \Pr\left[ \Hknp \rightarrow (F_1, \ldots, F_r) \right] = 1.
$$
\end{theorem}
The proof is a slight modification of the proof of Theorem~\ref{thm:asymmetric-1} and  we briefly sketch the argument. 
We exploit $\Hknp\not \rightarrow (F_1, \ldots, F_r)$ by applying the container  theorem, Theorem~\ref{thm:container}, as before with $F=F_i$ for $i\ge 2$ and with $H_i$ being the edges of $\Hknp$ coloured $i$ without a copy of $F_i$. The tuples $T^i$ consist then of sets of sizes at most $\ell n^{k - 1/m_k(F_i)}$ each. Thus, a colouring of the edges of $H\sim \Hknp$ which certifies $H\not \rightarrow (F_1, \ldots, F_r)$ allows us to place the $i$-coloured edges  $H_i$, for $i\ge 2$, into some container $C_i$ with fewer than $\alpha n^{v(F_i)}$ copies of $F_i$. Thus, Lemma~\ref{thm:ramsey_count} yields, that $H':=K^{(k)}_n\setminus \cup_{i\ge 2} E(C_i)$ contains at least $\alpha n^{v(F_1)}$ copies of $F_1$ for some absolute $\alpha>0$. Therefore, all edges in $H'\cap H$ are coloured in colour $1$ and since $H_1=H'\cap H$, none of the copies of $F_1$ from $H'$ is in $H_1$. 
But then Janson's inequality, viz.\ Lemma~\ref{lemma:janson}, yields the probability of at most $e^{-\beta C n^{k - 1/m_k(F_2)}\log n}$ that none of the copies of $F_1$ in $H_1$ is present in $H$. More precisely, let $F_1'$ be the subgraph of $F_1$ that is strictly balanced with respect to $m_k(\cdot,F_2)$. Then, Lemmas~\ref{lemma:funionf} and~\ref{lemma:janson} are applicable to $F_1'$ instead of $F_1$ and if $H_1$ doesn't contain a copy of $F'_1$ then it doesn't contain a copy of $F_1$ as well. The probability for the latter event is, by Lemma~\ref{lemma:janson}, at most $e^{-\beta C n^{k - 1/m_k(F_2)}\log n}$. The number of choices for $T^i$'s is at the same time at most 
\[
\binom{n^k}{r \ell n^{k - 1/m_k(F_2)}} 2^{r^2 \ell^2   n^{k - 1/m_k(F_2)}}=e^{O_{r,\ell}(n^{k - 1/m_k(F_2)}\log n)}.
\]
But then the union bound over all choices of $T^i$ finishes the claim, if we choose $C$ large enough, since the failure probability $e^{-\beta C n^{k - 1/m_k(F_2)}\log n}$ times the number of choices for $T^i$s is $o(1)$. It would be of interest to further remove the $\log n$ factor in Theorem \ref{thm:asymmetric-general}.

 Another research direction concerns sharp thresholds, already mentioned in the introduction. For example, in the case of a triangle and two colours,  $G(n,p)\rightarrow (K_3)$
 has a sharp threshold as proved by Friedgut, R\"odl, Ruci\'nski and Tetali~\cite{friedgut2006sharp}. Building on the work of Friedgut et.al.~\cite{FHPS16} on the sharpness of the threshold for the van der Waerden property in random subsets, Schacht and Schulenburg~\cite{schacht2016sharp} gave a proof for sharpness of a threshold for a class of strictly balanced and nearly bipartite graphs, i.e.\ those graphs that contain an edge whose deletion makes them bipartite. Such a class contains all odd cycles, and thus gives a shorter proof of the result in~\cite{friedgut2006sharp}. The new proof of the sharpness result uses an application of Friedgut's sharpness criterion~\cite{Fri05} in combination with container theorems~\cite{balogh2015independent,saxton2015hypergraph}. 
 It would be interesting to obtain sharpness results for some asymmetric graph Ramsey properties or even for hypergraphs.
 
In~\cite{nenadov2015algorithmic,thomas2013} we proved the $0$-statement for the threshold probability for $\Hknp\rightarrow (F)_r$, where $r\ge 2$ and $F$ is the complete $k$-uniform hypergraph. This $0$-statement is of the form $cn^{-1/m_k(F)}$ and thus it matches the corresponding bounds for the $1$-statement due to Conlon and Gowers~\cite{conlon2013combinatorial} and Friedgut, R\"odl and Schacht~\cite{friedgut2010ramsey} up to a multiplicative constant. In~\cite{thomas2013},  similar $0$-statements are proved for larger classes of hypergraphs apart from complete graphs. We refer the interested reader to Theorem~6.13 from~\cite{thomas2013}. However, we are far from having a complete characterization of the thresholds for $\Hknp\rightarrow (F)_r$ for general $F$ and $r\ge 2$. It would be interesting to decide whether the connection between asymmetric and symmetric Ramsey properties (Theorem~\ref{thm:main_beta}) occurs in the case of $3$-uniform hypergraphs as well.

\bibliographystyle{plain}
\bibliography{refs}

\newpage
\appendix


\section{Appendix}

\subsection{Proof of Lemma \ref{lemma:rdo_l4}}

First we show a simple claim, which is used  later in the proof.

\begin{claim}
\label{claim:xy_disjoint}
Let $H$ be a $3$-uniform hypergraph which does not contain a $1$-intersecting set of size three
and let $x,y \in E(H)$ be two disjoint edges of $H$. Then each edge $e \in E(H) \setminus \{x, y\}$ has to intersect
exactly one of the edges from $\{x,y\}$ and on exactly two vertices.
\end{claim}
\begin{proof}
Let $e$ be an arbitrary edge from $E(H) \setminus \{x, y\}$. If $e$ intersects both $x$ and $y$ on at most one vertex then set $\{x,y,e\}$ is a 1-intersecting set of size three, which is a contradiction with the assumption of the lemma. On the other hand, if $e$ intersects both $x$ and $y$ on two vertices, then $x$ and $y$ can not be disjoint.
\end{proof}

Next, we show  that Lemma \ref{lemma:rdo_l4} holds if $H$, additionally, does not contain a 1-intersecting set of size three.

\begin{lemma} \label{lemma:l4_special}
Let $H$ be a 3-uniform hypergraph with 7 edges and no 1-intersecting set of size three. There is a 1-intersecting subset $S \subseteq E(H)$ such that $H \setminus S$ does not contain a copy of $T^{3}_6$.
\end{lemma}
\begin{proof}
In order to obtain a contradiction, let us assume that for every 1-intersecting set $S \subseteq E(H)$ the graph $H \setminus S$ contains a $T_6^{3}$-copy. We show that this assumption gives us enough information to deduce important information about  the structure of  $H$.
Let $a$ and $b$ be two edges of an arbitrary $T_6^{3}$-copy of $H$ such that $a \cap b = \emptyset$. The set of edges $\{a, b\}$ is a 1-intersecting set and thus $H \setminus \{a, b\}$ contains a $T_6^{3}$-copy. Pick an arbitrary such copy and denote it by $F$. Let $E(F) = \{e_1, e_2, e_3, e_4\}$ be a natural order of edges of $F$ and let us denote the only remaining edge from $E(H) \setminus \{a, b , e_1, e_2, e_3, e_4\}$ by $e^*$. 
By applying Claim \ref{claim:xy_disjoint} to $a$ and $b$ we know that edges $e_1, e_4$ and $e^*$ have to intersect $a$ or $b$ on exactly two vertices. Note that $e_1$ and $e_4$ must not intersect the same edge on two vertices as then they would not be disjoint. Let us assume without loss of generality that
\begin{equation*}
|e_1 \cap a| = 2, \quad |e^* \cap a| = 2, \quad \text{and} \quad |e_4 \cap b| = 2. 
\end{equation*}
From previous observation it then directly follows that
\begin{equation}
\label{eq:obs1}
 |e_1 \cap b | \le 1, \quad |e^* \cap b | \le 1, \quad  \text{and} \quad |e_4 \cap a | \le 1.
\end{equation}
If $e^*$ intersects $e_4$ on two vertices, then $e^*$ intersects $e_1$ on at most one vertex and together with \eqref{eq:obs1} we conclude $\{e_1, b, e^*\}$ is a 1-intersecting set, which is a contradiction. 
Thus, we have
\begin{equation}
\label{eq:obs2}
|e^* \cap e_4| \le 1.
\end{equation}
Next, let us look at the edge $e_2$. If $|e_2 \cap a| \le 1$ then since $|a \cap e_4| \le 1$ by \eqref{eq:obs1} and by the definition of the natural order of the edges $E(F)$ from the path $T^3_6$ we conclude that $\{e_2, a, e_4\}$ is a 1-intersecting set, which is not possible. Thus, we know 
\begin{equation}
\label{eq:obs3}
|e_2 \cap a | = 2.
\end{equation}
Similarily, by \eqref{eq:obs2} and the fact that $|e_2 \cap e_4| =1 $ we obtain
\begin{equation}
\label{eq:obs4}
|e_2 \cap e^* | = 2.
\end{equation}
To finish the proof, it suffices to show that $H \setminus \{e_1, e_3\}$ does not contain a copy of $T_6^{3}$. 
Let $\tilde F$ be an arbitrary $T_6^{3}$-copy from $H \setminus \{e_1, e_3\}$. As there are only five edges in $H \setminus \{e_1, e_3\}$ either both $e_2$ and $b$  or both $e^*$ and $a$ are contained in $E(\tilde F)$. We split the rest of the proof into these two cases:

\paragraph{Case 1:} $\{e_2, b\} \subseteq E(\tilde F)$ \\
By \eqref{eq:obs3} we know $|e_2 \cap b| \le 1$, but we show that actually $|e_2 \cap b| = 1$.
It must be that $|e_3 \cap b| = 2$ as otherwise by \eqref{eq:obs1} the set $\{e_3, b, e_1\}$ is 1-intersecting.
Since $|e_3 \cap b| = 2$ and $|e_3 \cap e_2| = 2$ we get $|e_2 \cap b| \ge 1$ and thus $|e_2 \cap b| = 1$.
By applying Lemma \ref{lemma:core_l_r} to $\tilde F$, $e_2$ (as $a_0$) and $b$ (as $a_1$) we conclude that there  exists an edge
$x \in \{e_4, a, e^*\}$ such that $|x \cap e_2| = 2$ and $|x \cap b| = 2$. 
As $a \cap b = \emptyset$, $|e_4 \cap e_2| = 1$ and $|e^* \cap b| \le 1$ we know that $x \notin \{e_4, a, e^*\}$, which is a contradiction.

\paragraph{Case 2:} $\{e^*, a\} \subseteq E(\tilde F)$ \\
Since $|e^* \cap a| = 2$
by applying Lemma \ref{lemma:core_l_r} to $\tilde F$, $e^*$ (as $a_0$) and $a$ (as $a_1$) we conclude that there  exist an edge 
$x \in \{b, e_2, e_4\}$
such that $|x \cap a | + |x \cap e^*| = 3$.
As $b \cap a = \emptyset$ it can not be that $x = b$. Furthermore, using \eqref{eq:obs3} and \eqref{eq:obs4} we know that $|e_2 \cap a| + |e_2 \cap e^*| = 4$ and thus $x \neq e_2.$ Finally, using \eqref{eq:obs1} and \eqref{eq:obs2} we know $|e_4 \cap a| + |e_4 \cap e^*| \le 2$ and thus $x \neq e_4.$ This concludes the proof.
\end{proof}

Finally, by using lemma above we prove Lemma \ref{lemma:rdo_l4}

\begin{proof}[Proof of Lemma \ref{lemma:rdo_l4}]
If $H$ does not contain a 1-intersecting set of size three, then the lemma follows directly by Lemma \ref{lemma:l4_special}. Let us assume that $H$ contains a 1-intersecting set of size three $\{a, b, c\}$. Moreover, in order to arrive at contradiction we assume that for every 1-intersecting set $S \subseteq E(H)$ there exist a $T^{3}_6$-copy 
contained in $H \setminus S$.
Consequently, as $|E(H) \setminus \{a,b,c\}| = 4$, the hypergraph $H \setminus \{a, b, c \}$ must be isomorphic to $T^{3}_6$. Let $\{v_1, \ldots, v_6\}$ and $\{e_1, e_2, e_3, e_4\}$ be a natural order of vertices and edges of the $T^{3}_6$-copy spanned by $|E(H) \setminus \{a, b, c\}|$ such that
$
e_i = \{v_i, v_{i+1}, v_{i+2}\}.
$
We claim that either the edge $e_1$ or $e_4$ must be 1-intersecting with at least two edges from $\{a, b, c\}$. 
If we assume otherwise, then  there must be an edge in $\{a, b, c \}$, say $a$, such that
both $e_1$ and $e_4$ intersect $a$ on two vertices. However, that is not possible as $e_1 \cap e_4 = \emptyset$.
Without loss of generality, let us assume  $\{e_1, a, b\}$ is a 1-intersecting set. If $e_1$ is 1-intersecting with all the edges from $\{a, b, c\}$, then $H \setminus \{e_1, a, b, c\}$  contains only three edges and can not contain a $T^{3}_6$-copy, thus a contradiction.  In the rest of the proof we consider the case when $|e_1 \cap c| = 2$.

Since $\{e_1, a, b\}$ is a 1-intersecting set we know that $\{c,  e_2, e_3, e_4\}$ induces a copy of $T^{3}_6$.
Furthermore, the assumption $|e_1 \cap c| = 2$ implies $|c \cap e_2| \ge 1$. From $|e_1\cap e_4|=0$ we have $|c\cap e_4|\le 1$. If $|c\cap e_4|=1$ and since  $\{c,  e_2, e_3, e_4\}$ induces a copy of $T^{3}_6$, it follows that 
$|c\cap e_3|=0$ and $|e_2\cap e_4|=2$, which is a contradiction to $|e_2\cap e_4|=1$. 
Therefore we obtain
\begin{equation}
|c \cap e_4| = 0 \; \text{and} \; |c \cap e_2| = 2. \label{eq:xc} 
\end{equation}

Next, as $e_2$ and $e_4$ are 1-intersecting we know $H \setminus \{e_2, e_4\}$ contains a $T^{3}_6$-copy, denoted by $F$. To conclude the proof, we show that this is impossible.
Let us now consider the following two cases:
\vspace{-4mm}
\paragraph{Case 1:} $\{e_1, e_3\} \not \subseteq E(F)$

Since at most one edge from $\{e_1, e_3\}$ is part of $F$ we conclude $\{a, b, c\} \subseteq E(F)$.
This is a contradiction as by Observation \ref{claim:tl_comp_set} we know that $F$ can not contain a 1-intersecting set of size three.
\vspace{-4mm}
\paragraph{Case 2:} $\{e_1, e_3\} \subseteq E(F)$

Using Lemma \ref{lemma:core_l_r} we conclude that one of the edges from $E(F) \setminus \{e_1, e_3\}$ intersects both $e_1$ and $e_3$ on two vertices. Neither $a$ or $b$ can be such edge as they are 1-intersecting with $e_1$. On the other hand, we know $|c \cap e_4| = 0$ which implies $|c \cap e_3| \le 1$. 
Therefore no edge from set $E(F) \setminus \{e_1, e_3\}$ intersects both $e_1$ and $e_3$ on two vertices, which is a contradiction.
\end{proof}

\end{document}